\documentclass{tac}

\usepackage{graphicx}
\usepackage{amsmath,amssymb}
\usepackage[all]{xy}

\xyoption{2cell}
\UseAllTwocells

\usepackage{tikz}
\usetikzlibrary{arrows,calc,decorations,decorations.markings,fadings,positioning,patterns,shapes
,decorations.pathmorphing,backgrounds,fit,petri,mindmap,shadows,positioning,matrix}
\usepackage{multirow,tabularx,makecell}

\newcommand{\inv}{^{-1}}

\newtheorem{theorem}{Theorem}
\newtheorem{definition}[theorem]{Definition}
\newtheorem{proposition}[theorem]{Proposition}
\newtheorem{lemma}[theorem]{lemma}

\newcommand{\wquot}{/\!\!/}

\newcommand{\ra}{\rightarrow}
\newcommand{\act}{\blacktriangleright}

\address{Department of Electrical Engineering,\\
	Indian Institute of Technology Delhi, \\
	New Delhi-110016, INDIA.\\[5pt]
	Department of Electrical Engineering,\\
	Indian Institute of Technology Delhi, \\
	New Delhi-110016, INDIA.\\
}
\eaddress{salil.samant@ee.iitd.ac.in\CR sdjoshi@ee.iitd.ac.in}

\title{Unified Functorial Signal Representation II: Category action, Base Hierarchy, Geometries as Base structured categories}
\author{Salil Samant and Shiv Dutt Joshi}

\begin{document}
	
\maketitle
\begin{abstract}
In this paper we propose and study few applications of the base structured categories $\mathcal{X} \rtimes_{\mathbf{F}} \mathbf{C}$, $\int_{\mathbf{C}} \bar{\mathbf{F}}$, $\mathcal{X} \rtimes_{\mathbb{F}} \mathbf{C}$ and ${\int_{\mathbf{C}} \bar{\mathbb{F}}}$. First we show classic transformation groupoid $X \wquot G$ simply being a base-structured category ${\int_{\mathbf{G}} \bar{{F}}}$. Then using permutation action on a finite set, we introduce the notion of a hierarchy of base structured categories $[(\mathcal{X}_{2a} \rtimes_{\mathbf{F_{2a}}} \mathbf{B}_{2a}) \amalg (\mathcal{X}_{2b} \rtimes_{\mathbf{F_{2b}}} \mathbf{B}_{2b}) \amalg ...] \rtimes_{\mathbf{F_{1}}} \mathbf{B}_1$ that models local and global structures as a special case of composite Grothendieck fibration. Further utilizing the existing notion of transformation double category $(\mathcal{X}_{1} \rtimes_{\mathbf{F_{1}}} \mathbf{B}_{1}) \wquot \mathbf{2G}$, we demonstrate that a hierarchy of bases naturally leads one from 2-groups to n-category theory. Finally we prove that every classic Klein geometry is the Grothendieck completion ($\mathbf{G} = \mathcal{X} \rtimes_{\mathbb{F}} \mathbf{H}$) of ${\mathbb{F}}: \mathbf{H} \xrightarrow{{F}} \mathbf{Man}^{\infty} \xrightarrow{U} \mathbf{Set}$. This is generalized to propose a set-theoretic definition of a groupoid geometry $(\mathcal{G},\mathcal{B})$ (originally conceived by Ehresmann through transport and later by Leyton using transfer) with a principal groupoid $\mathcal{G} = \mathcal{X} \rtimes \mathcal{B}$ and geometry space $\mathcal{X} = \mathcal{G}/\mathcal{B}$; which is essentially same as $\mathbf{G} = \mathcal{X} \rtimes_{\mathbb{F}} \mathbf{B}$ or precisely the completion of ${\mathbb{F}}: \mathbf{B} \xrightarrow{{F}} \mathbf{Man}^{\infty} \xrightarrow{U} \mathbf{Set}$.  

%
%

\end{abstract}

\section{Introduction}
\label{intro}

The notion of base structured categories characterizing a functor were introduced in the preceding part \cite{salilp1} of the sequel. Here the perspective of a functor as an action of a category is considered by using the elementary structure of symmetry on objects especially the permutation action on a finite set. The geometric structure of symmetry is considered to be a fundamental notion in all pure and applied branches of sciences including physics and information science. Historically the algebraic structure of a group has been associated with symmetry \cite{armstrong}. Especially in the representation of signals carrying information structures, groups have played a key role in capturing the symmetry on the signal space. Meanwhile from the mid 20th century a new algebraic structure of a groupoid was found more appropriate for modeling of symmetries on account of its ability to authentically capture both local and global symmetry or in other words internal and external symmetries; see both \cite{weinstein} and \cite{brown87} as well as references therein for systematic account of advantages of groupoids over groups in various broad contexts such as algebraic topology, algebraic geometry, differential topology and geometry, analysis in addition to simple symmetry structure. 

In this paper we study and note the precise difference between a group and groupoid through the perspective of base structured category and fibration of objects using an example of simple set permutation. Symmetry of a hierarchical structure is also studied using single object, multi level bases as well as multi object, multi level bases. The algebraic structure called groupoid naturally and elegantly models phenomena involving the concept of local and global or internal and external symmetries of objects. It is well-known in category theory that the category of pointed and connected groupoids is equivalent to the category of groups implying precisely that it is the non-pointedness and the non-connectedness in general of groupoids over groups which holds the key in all such situations involving some hierarchy of symmetries. We intuitively also relate the imprimitive action of wreath product group with the natural groupoid action on the same simple object and precisely note the key differences due to crucial variability in the pointedness and connectedness of these structures. This motivated us to consider a notion of multi-object species of structures using the concept of multiple objects in the base. 

\subsection{Some Background and Related Work}
Groupoid theory in the context of symmetries is illustrated using numerous examples in \cite{weinstein} and the categorification of local and global symmetries through the central notion of transformation double category is considered in \cite{morton}. The historic perspective of shift from groups to groupoids along with an overview of applications is summarized in \cite{brown87} while groupoids in the context of topological spaces are thoroughly studied in \cite{TGD}. Groupoids in the context of differential geometry were considered much earlier during the past century in \cite{charles} and reviewed in \cite{pradines}. 

In \cite{Leyton01}, Micheal Leyton utilized wreath group action for modeling maximal transfer involved in human psychological perception of auditory structures and visual shapes in psychology. Part of the motivation for this paper and the notion of natural generators as utilized in the sequel comes from the way of modeling generators in \cite{Leyton01}. Using the the work of \cite{charles} we realized that the intuitive notion of transfer in \cite{Leyton01} is same as the rigorous transport of structure formulated by Charles Ehresmann and illustrated in \cite{brown87}.

The notion of species of structures is briefly reviewed in \cite{brown87} while the notion in connection with combinatorics is referred from \cite{joyalspecies}. Looking at the vast field of differential geometry \cite{sharpe} and n-category theory \cite{baeztowards},\cite{leinster} at the moment, the full connections and possible ramifications (if any) of the notion of base hierarchy and base structured geometry on these fields can be perceived only in the future work.

\subsection{Contribution}

In Section~\ref{gc} we briefly revisit the base structured categories $\mathcal{X} \rtimes_{\mathbf{F}} \mathbf{C}$ (denoting abstract left action) $\mathcal{X} \rtimes_{\mathbb{F}} \mathbf{C}$ (denoting concrete left action) as well as their right action counterparts introduced in the preceding part \cite{salilp1} which characterizes a functor as a multi-object action of a category. Motivating from a simple abstract definition of a group action; a similar abstract definition is formulated as in~\ref{def:cat_action1}. Such an abstract form encompasses many equivalent definitions which are based on particular examples of objects such as sets, topological spaces in the corresponding category $\mathbf{D}$ as well as special cases of the $\mathbf{C}$ category such as a groupoid all in a unified way. In section~\ref{sec:basesymm} we prove that a contemporary transformation groupoid is nothing but precisely a base structured category and study few connections with symmetry. 

Next in Sections~\ref{sec:wreath} and~\ref{h} using simple example of permutation action on a finite set at two levels (local and global) we review the essence of Leyton's theory of structured transfer using the notion of hierarchy of base structured categories (where the generative fiber object is essentially transferred by the morphisms of the underlying base category using the perspective of a functor as category action). The total (base structured) category captures the new structure (formed out of the controlled transfer of object another category), being fibered over a base category, capturing the control structure. The special case of wreath group as a (single-object 2-level) base and semi-direct groupoid (as multi-object 2-level) base is also naturally recovered by using the notion of hierarchy of bases in base structured categories. Finally it is shown that hierarchy of base structured categories is just a special case within the theory of composite Grothendieck fibrations.

Further in Section~\ref{2g} we utilize the notion of 2-group action on a category from \cite{morton} to propose a hierarchy of 2-groups capturing the symmetry of hierarchy of base structured categories naturally leading into n-groups or higher category theory. This resembles to us a multi-resolution wavelet analysis common in signal processing or machine learning community.

Finally we propose that every geometry is a base structured category and formulate proof for classic Klein gemeotries. Then as a generalization of Klein geometry we present a simple definition~\ref{geometry} of groupoid geometry from the concept of category action taking motivation from work of Charles Ehresmann \cite{charles} and Jean Pradines \cite{pradines} where the term groupoid geometry was first coined.

\subsection{Acknowledgments}
Many commutative diagrams were produced using Kristoffer Rose's XY-pic package and rest using standard TikZ package.

\section{From group action to category action}
\label{gc}

Since this paper exclusively deals with the characterization of a functor as a category action, we begin with a simple group action and illustrate using the intuitive visual approach of \cite{VGT}, the notion of a category action through visual diagrams.

Recall that there are two different but equivalent definitions of a simple group action. First is the concrete, set-theoretic (object-based) definition of a group action;
\begin{definition} \cite{maclane}
	\label{def:gr_action1}
	Let $X$ be a set and $G$ be a group then a (left) group action is a function $\mu:G \times X \rightarrow X$ such that
	\begin{enumerate}
		\item $\mu(g_2, \mu(g_1,x)) = \mu(g_2g_1,x)$
		\item $\mu(id,x) =x$
	\end{enumerate}
\end{definition}

The second definition is abstract or arrow-based that defines same action as a permutation representation. This is an authentic arrow definition that follows the common arrow philosophy of category theory.

\begin{definition} \cite{maclane}
	\label{def:gr_action2}
	A permutation representation of $G$ on $X$ is a homomorphism from $G$ into $\mathsf{S}(X)$ where $\mathsf{S}(X)$ is the symmetry or automorphism group of $X$.
\end{definition} 

The following Lemma~\ref{lm:ga_prem} establishes the equivalence between these definitions.
\begin{lemma}\cite{maclane}
	\label{lm:ga_prem}
	If $G$ is a group and $X$ is a set then,
	\begin{enumerate}
		\item If $\mu: G \times X \rightarrow X$ is a group action then the map $\phi: G \mapsto \mathsf{S}(X)$ defined as $\phi(g)(x) = \mu (g,x)$ is a permutation representation of $G$ on $\mathsf{S}(X).$
		\item If $\phi:G \mapsto \mathsf{S}(X)$ is a permutation representation then a group action of $G$ on $X$ as is defined as $\mu(g,x):= \phi(g)(x)$.
	\end{enumerate}
\end{lemma}

The proof is left to the reader and can be found in \cite{armstrong}. Now observe that the group homomorphism is same as functor $F:\mathbf{G} \rightarrow \mathbf{Set}$ where $\mathbf{G}$ is a category with a single object $\star$, and homset $\mathbf{G}(\star,\star) = G$ with $F(\star) = X$. This immediately motivates one to generalize this to a functor $F:\mathbf{C} \rightarrow \mathbf{D}$. This is precisely what is meant by an action of a category as defined and utilized in this sequel. The analogy immediately suggests an abstract Definition~\ref{def:cat_action1}.

\begin{definition}(Abstract or arrow theoretic)
	\label{def:cat_action1}
	A transformation representation of a category $\mathbf{C}$ on $\mathcal{X}$ is a functor from $\mathbf{C}$ into a category $\mathbf{D}$ where $\mathcal{X} = \amalg_{X \in Ob(\mathbf{C})} F(X)$ is the object of objects (or a coproduct of objects) in $\mathbf{D}$.
\end{definition}

It is assumed that such a coproduct of the codomain category $\mathbf{D}$ exists which implies that such a representation induces an action on the multiple objects of $F(\mathbf{C})$. Such an abstract (left and right) action gives rise to the base structured categories $\mathcal{X} \rtimes_{\mathbf{F}} \mathbf{C}$ and $\int_{\mathbf{C}} \bar{\mathbf{F}}$ as defined in~\cite{salilp1} of this sequel and recalled again here.

\begin{definition}[Abstract Left action induced by a functor]~\cite{salilp1}
	Consider a (covariant) functor $F: \mathbf{C} \rightarrow \mathbf{D}$ between small categories thought of as covariant $\mathbf{F}: \mathbf{C} \rightarrow \mathbf{Cat}$ (with $\mathbf{F} = I \circ {F}$ and $I:\mathbf{D} \rightarrow \mathbf{Cat}$ as defined). Then $\mathcal{X} \rtimes_{\mathbf{F}} \mathbf{C}$ (or $(\int_{\mathbf{C}^{op}} \bar{\mathbf{F}})^{op}$) is a category with 
	
	\begin{itemize}
		\item \textbf{objects}: the pairs $(X,FX)$ where $X \in \mbox{Ob}(\mathbf{C})$ and $FX \in \mbox{Ob}(\mathbf{D})$
		
		\item \textbf{morphisms}: $\mathcal{X} \rtimes_{\mathbf{F}} \mathbf{C}((X,FX),(Y,FY))$ are pairs $(f,\mathrm{id}_{FY})$ where $f:X \rightarrow Y \in \mathbf{C}$, $\mathrm{id}_{FY}:\mathbf{F}f(FX) \rightarrow FY$
		
		\item \textbf{identity}: for $(X,FX)$, the morphism $\mathrm{id}_{(X,FX)} = (\mathrm{id}_{X},\mathrm{id}_{FX})$
		
		\item \textbf{composition}: $(g,\mathrm{id}_{FZ}) \bullet (f,\mathrm{id}_{FY}) = (g \circ f,\mathrm{id}_{FZ} \cdot \mathbf{F}g(\mathrm{id}_{FY})) = (gf,\mathrm{id}_{FZ})$ since $\mathbf{F}g(\mathrm{id}_{FY}) = \mathbf{F}g\mathbf{F}f(FX) \rightarrow \mathbf{F}g(FY)$ 
		
		\item \textbf{unit laws}: for $(f,\mathrm{id}_{FY})$, 
		$(\mathrm{id}_{Y},\mathrm{id}_{FY}) \bullet (f,\mathrm{id}_{FY}) = (f,\mathrm{id}_{FY}) = (f,\mathrm{id}_{FY}) \bullet (\mathrm{id}_{X},\mathrm{id}_{FX})$    
		
		\item \textbf{associativity}:  $(h,\mathrm{id}_{FW})\bullet ((g,\mathrm{id}_{FZ})\bullet (f,\mathrm{id}_{FY}))=((h,\mathrm{id}_{FW})\bullet (g,\mathrm{id}_{FZ}))\bullet (f,\mathrm{id}_{FY}) =(h,\mathrm{id}_{FW})\bullet (g,\mathrm{id}_{FZ})\bullet (f,\mathrm{id}_{FY})$
	\end{itemize}
	
\end{definition} 

\begin{definition}[Abstract Right action induced by a functor]~\cite{salilp1}
	\label{defn3}
	Consider a strict contravariant functor $\bar{F}: \mathbf{C} \rightarrow \mathbf{D}$ between small categories thought of as $\bar{\mathbf{F}}: \mathbf{C} \rightarrow \mathbf{Cat}$ (with $\bar{\mathbf{F}} = I \circ \bar{F}$ and $I:\mathbf{D} \rightarrow \mathbf{Cat}$ as defined). Then $\int_{\mathbf{C}} \bar{\mathbf{F}}$ is a category with

	\begin{itemize}
		\item \textbf{objects}: the pairs $(X,\bar{F}X)$ where $X \in \mbox{Ob}(\mathbf{C})$ and $\bar{F}X \in \mbox{Ob}(\mathbf{D})$
		
		\item \textbf{morphisms}: ${\int_{\mathbf{C}} \bar{\mathbf{F}}}((X,\bar{F}X),(Y,\bar{F}Y))$ are pairs $(f,\mathrm{id}_{\bar{F}X})$ where $f:X \rightarrow Y \in \mathbf{C}$, $\mathrm{id}_{\bar{F}X}:\bar{F}X \rightarrow \bar{\mathbf{F}}f(\bar{F}Y)$
		
		\item \textbf{identity}: for $(X,\bar{F}X)$, the morphism $\mathrm{id}_{(X,\bar{F}X)} = (\mathrm{id}_{X},\mathrm{id}_{\bar{F}X})$
		
		\item \textbf{composition}: $(g,\mathrm{id}_{\bar{F}Y}) \bullet (f,\mathrm{id}_{\bar{F}X}) = (g \circ f,\bar{\mathbf{F}}f(\mathrm{id}_{\bar{F}Y}) \cdot \mathrm{id}_{\bar{F}X}) = (gf,\mathrm{id}_{\bar{F}X})$ since $\bar{\mathbf{F}}f(\mathrm{id}_{\bar{F}Y}): \bar{\mathbf{F}}f\bar{F}Y \rightarrow \bar{\mathbf{F}}f\bar{\mathbf{F}}g(\bar{F}Z)$ 
		
		\item \textbf{unit laws}: for $(f,\mathrm{id}_{\bar{F}X})$, 
		$(\mathrm{id}_{Y},\mathrm{id}_{\bar{F}Y}) \bullet (f,\mathrm{id}_{\bar{F}X}) = (f,\mathrm{id}_{\bar{F}X}) = (f,\mathrm{id}_{\bar{F}X}) \bullet (\mathrm{id}_{X},\mathrm{id}_{\bar{F}X})$    
		
		\item \textbf{associativity}:  $(h,\mathrm{id}_{\bar{F}Z})\bullet ((g,\mathrm{id}_{\bar{F}Y})\bullet (f,\mathrm{id}_{\bar{F}X}))=((h,\mathrm{id}_{\bar{F}Z})\bullet (g,\mathrm{id}_{\bar{F}Y}))\bullet (f,\mathrm{id}_{\bar{F}X}) =(h,\mathrm{id}_{\bar{F}Z})\bullet (g,\mathrm{id}_{\bar{F}Y})\bullet (f,\mathrm{id}_{\bar{F}X})$
	\end{itemize}
\end{definition}

The abstract definition can be made concrete or set-theoretic by considering small category and its action on the objects as structured sets and morphisms as structure preserving functions. Thus concrete (left and right) action gives rise to the base structured categories $\mathcal{X} \rtimes_{\mathbb{F}} \mathbf{C}$ and $\int_{\mathbf{C}} \bar{\mathbb{F}}$ as defined in~\cite{salilp1} of this sequel which we recall here.

\begin{definition}[Concrete Left action induced by a functor]
	\label{defn6}
	Consider a covariant functor ${F}: \mathbf{C} \rightarrow \mathbf{D}$ between small categories with $(\mathbf{D}, U)$ being a concrete category over $\mathbf{Set}$ or a faithful $U:\mathbf{D} \rightarrow \mathbf{Set}$. Then ${\mathbb{F}} = U \circ {F}$ and $\mathcal{X} \rtimes_{\mathbb{F}} \mathbf{C}$ (or $(\int_{\mathbf{C}^{op}} \bar{\mathbb{F}})^{op}$) is a category with

	\begin{itemize}
		\item \textbf{objects}: the pairs $(X,x)$ where $X \in \mbox{Ob}(\mathbf{C})$ and $x \in {\mathbb{F}}X = U({F}X)$
		
		\item \textbf{morphisms}: pairs $(f,y): (X,x) \rightarrow (Y,y)$ where $f:X \rightarrow Y \in \mathbf{C}$, $y = {\mathbb{F}}f(x)$
		
		\item \textbf{identity}: for $(X,x)$, the morphism $\mathrm{id}_{(X,x)} = (\mathrm{id}_{X},x)$
		
		\item \textbf{composition}: $(g,{z}) \bullet (f,{y}) = (g \circ f,z \cdot \mathbb{F}g(y)) = (g \circ f,z)$ since $z = {\mathbb{F}}g{\mathbb{F}}f(x)$ 
		
		\item \textbf{unit laws}: for $(f,{y})$, 
		$(\mathrm{id}_{Y},{y}) \bullet (f,{y}) = (f,{y}) = (f,{y}) \bullet (\mathrm{id}_{X},{x})$    
		
		\item \textbf{associativity}:  $(h,{w})\bullet ((g,{z})\bullet (f,{y}))=((h,{w})\bullet (g,{z}))\bullet (f,{y}) =(h,{w})\bullet (g,{z})\bullet (f,{y})$
	\end{itemize}
\end{definition}

\begin{definition}[Concrete Right action induced by a functor]
	\label{defn5}
	Consider a strict contravariant functor $\bar{F}: \mathbf{C} \rightarrow \mathbf{D}$ between small categories with $(\mathbf{D}, U)$ being a concrete category over $\mathbf{Set}$ or a faithful $U:\mathbf{D} \rightarrow \mathbf{Set}$. Then ${\bar{\mathbb{F}}} = U \circ \bar{F}$ and $\int_{\mathbf{C}} \bar{\mathbb{F}}$ is a category with

	\begin{itemize}
		\item \textbf{objects}: the pairs $(X,x)$ where $X \in \mbox{Ob}(\mathbf{C})$ and $x \in \bar{\mathbb{F}}X = U(\bar{F}X)$
		
		\item \textbf{morphisms}: pairs $(f,x): (X,x) \rightarrow (Y,y)$ where $f:X \rightarrow Y \in \mathbf{C}$, $x = \bar{\mathbb{F}}f(y)$
		
		\item \textbf{identity}: for $(X,x)$, the morphism $\mathrm{id}_{(X,x)} = (\mathrm{id}_{X},x)$
		
		\item \textbf{composition}: $(g,{y}) \bullet (f,{x}) = (g \circ f,{x})$ since $x = \bar{\mathbb{F}}f\bar{\mathbb{F}}g(z)$ 
		
		\item \textbf{unit laws}: for $(f,{x})$, 
		$(\mathrm{id}_{Y},{y}) \bullet (f,{x}) = (f,{x}) = (f,{x}) \bullet (\mathrm{id}_{X},{x})$    
		
		\item \textbf{associativity}:  $(h,{z})\bullet ((g,{y})\bullet (f,{x}))=((h,{z})\bullet (g,{y}))\bullet (f,{x}) =(h,{z})\bullet (g,{y})\bullet (f,{x})$
	\end{itemize}
\end{definition} 

	\begin{figure}[!h]
	\begin{equation*}
	\xymatrix{
		(X,x)   &  & (Y,y)  & & (Z,z) \\
		(X,x) \ar[u]^{(id_X,x)} \ar@{.>}[rr] \ar[rrrru]^{{(g \circ f,z)}} & & (Y,\mathbb{F}f(x)) \ar@{.>}[rr]  & & (Z,\mathbb{F}g\mathbb{F}f(x)) \ar[u]_{(id_Z,z)} \\
		(X,x)   &  & (Y,y)  & & (Z,z) \\
		(X,x) \ar[u]^{(id_X,x)} \ar@{.>}[rr]^{U(Ff)}  \ar@{>}[rru]^{(f,y)} & & (Y,\mathbb{F}f(x)) \ar[u]^{(id_Y,y)} \ar@{.>}[rr]^{U(Fg)} \ar@{>}[rru]^{(g,z)} & & (Z,\mathbb{F}g(y)) \ar[u]_{(id_Z,z)} \\
		X \ar@/^1pc/[rrrr]^{g \circ f} \ar@(dl,ul)[]|{id_X} \ar[rr]_{f} &  & Y \ar@(dr,ur)[]|{id_Y} \ar[rr]_{g} & &  Z \ar@(dr,ur)[]|{id_Z}
	}
	\end{equation*}
	\caption{$\mathcal{X} \rtimes_{\mathbb{F}} \mathbf{C}$ fibred on $\mathbf{C}$; dotted arrows denote concrete functions as left actions}
	\label{fig:leftconcretefib}
\end{figure}
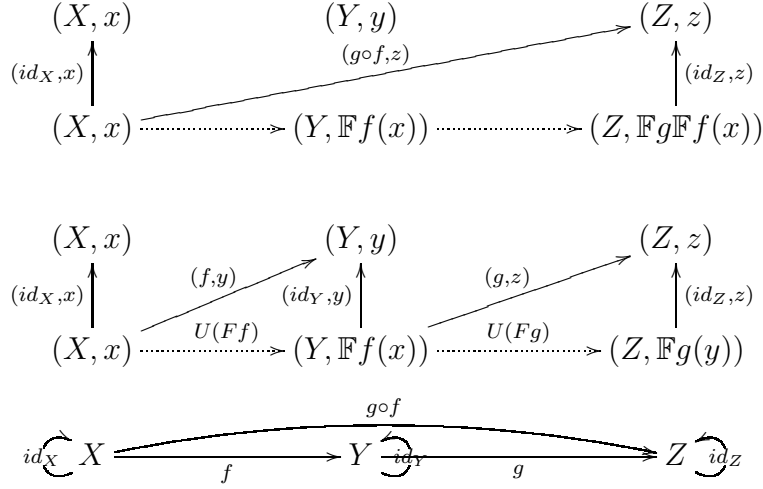 

	\begin{figure}[!h]
	\begin{equation*}
	\xymatrix{
		(X,\bar{\mathbb{F}}f\bar{\mathbb{F}}g(z))  &  & (Y,\bar{\mathbb{F}}g(z)) \ar@{.>}[ll]  & & (Z,z) \ar@{.>}[ll] \\
		(X,x) \ar[u]^{(id_X,{x})} \ar@{>}[rrrru]^{(g \circ f,{x})} & & (Y,y)  & & (Z,z) \ar[u]_{(id_Z,{z})} \\
		(X,\bar{\mathbb{F}}f(y))  &  & (Y,\bar{\mathbb{F}}g(z)) \ar@{.>}[ll]_{U(\bar{F}f)}  & & (Z,z) \ar@{.>}[ll]_{U(\bar{F}g)} \\
		(X,x) \ar[u]^{(id_X,{x})}  \ar@{>}[rru]^{(f,{x})} & & (Y,y) \ar[u]^{(id_Y,{y})} \ar@{>}[rru]^{(g,{y})} & & (Z,z) \ar[u]_{(id_Z,{z})} \\
		X \ar@/^1pc/[rrrr]^{g \circ f} \ar@(dl,ul)[]|{id_X} \ar[rr]_{f} &  & Y \ar@(dr,ur)[]|{id_Y} \ar[rr]_{g} & &  Z \ar@(dr,ur)[]|{id_Z}
	}
	\end{equation*}   
	\caption{$\int_{\mathbf{C}}\bar{\mathbb{F}}$ fibred on $\mathbf{C} \cong \mathbf{C}^{op}$; dotted arrows denote concrete functions as right actions}
	\label{fig:rightconcretefib}
\end{figure}
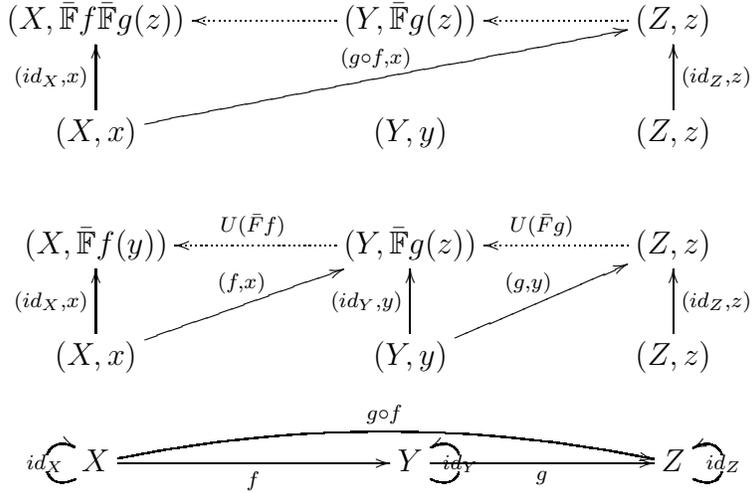

\subsection{Sets as Categories}
In category theory sets are often viewed as categories. More precisely a category is a set if all morphisms are identity morphisms and referred to as \textbf{discrete category}. This is possible using a standard inclusion $\iota : \mathbf{Set} \rightarrow \mathbf{Cat}$. Thus every contravariant functor thought of as a covariant functor (or a presheaf) $F: \mathbf{C}^{op} \rightarrow \mathbf{Set}$ can be completed to a fibred category which is said to be fibred in sets. Thus the base structured categories which arise through construction or completion of functors such as $\bar{\mathbb{F}}: \mathbf{C}^{op} \xrightarrow{\bar{F}} \mathbf{D} \xrightarrow{U} \mathbf{Set}$ are truly nothing but completion of $\bar{\mathbb{F}}':\mathbf{C}^{op} \xrightarrow{\bar{F}} \mathbf{D} \xrightarrow{U} \mathbf{Set} \xrightarrow{\iota} \mathbf{Cat}$ which is $\int_{\mathbf{C}}\iota \bar{\mathbb{F}}$ where $\iota \bar{\mathbb{F}} = \bar{\mathbb{F}}'$. However we will suppress the notation $\iota$ unless explicitly required within a context since Grothendieck completion of $\mathbf{Set}$-valued functors is well studied (via the understood inclusion).

The set-theoretic groupoid based definitions are adapted from \cite{gcp} for the general category action.  

\begin{definition}\cite{gcp}
	\label{def:cat}
	Let $C$ be a set and $C^{(2)}\subseteq C\times C$.  Then $C$ is a small category if there are maps $(g,f)\mapsto gf$ from $C^{(2)}$ into $C$ such that:
	\begin{itemize}
		\item \textbf{associativity}: If $(h,g)$ and $(g,f)$ belong to $C^{(2)}$ then
		$(hg,f)$ and $(h,gf)$, also belong to $C^{(2)}$ and we have $(hg)f = h(gf)$.  
		
		\item \textbf{identity}: the set of elements with $(f,f) = f$ denoted as $C^{(0)}$  
		
		\item \textbf{unit laws}: If $f \in C$ then $(r(f),f),(f,s(f))\in
		C^{(2)}$, $r(f)f = f$, and $fs(f) = f$.  
	\end{itemize}
	where the map $r:C\rightarrow C^{(0)}$ such that $r(f) = ff\inv$ is called the {\em range map} and the map $s:C\rightarrow C^{(0)}$ such that $s(f) = f\inv f$ is called the {\em source map}.
	when $(g,f)\in C^{(2)}$ then $g$ and $f$ are called {\em composable}, or the set $C^{(2)}$ is the set of {\em composable pairs}. In the case when we have an additional map $f \mapsto f\inv$ from $C$ into $C$ such that for every $f \in C$ we have $(f\inv)\inv = f$, then such a category is called a groupoid and $f\inv$ is the inverse of $f$.
\end{definition}

Note that this definition is equivalent to the one we have utilized in \cite{salilp1} with objects as elements belonging to $C^{(0)}$ while the morphisms are elements of $C$. Corresponding to the Definition~\ref{def:cat} we can now define a category action on a set. 

\begin{definition}\cite{gcp}
	\label{def:cat_action2}
	Suppose $C$ is a small category and $X$ is a set. Then a category $C$ action
	(left) on $X$, is a surjection $r_X:X \rightarrow C^{(0)}$ and a map $(g,x)\mapsto
	g \cdot x$ from $C*X:=\{(f,x)\in C\times X:
	s(f)=r_X(x)\}$ to $X$ such that 
	\begin{enumerate}
		\item if $(f,x)\in C*X$ and $(g,f)\in C^{(2)}$, then 
		$(gf,x),(g,f\cdot x)\in C*X$ and 
		\[
		g\cdot(f\cdot x) = gf \cdot x,
		\]
		\item and $r_X(x) \cdot x = x$ for all $x\in X$. 
	\end{enumerate}
	In the same manner for the right action we use $s_X$ to denote the map from $X$ to $C^{(0)}$ and then define the action on the set $X*C := \{(x,f)\in X\times C: s_X(x) = r(f)\}$.
\end{definition}

However going further we will not utilize this form of a category definition and the action especially in this paper; since it obscures the underlying arrow approach of the whole category theory by treating it as a set with some special set of axioms.

\section{Base Structured Categories in Symmetry}
\label{sec:basesymm}
By relating the contemporary action or transformation groupoid to base structured category we gain an added perspective in the context of symmetries especially regarding the differences between local and global symmetry in basic examples and also cases involving some sort of symmetry hierarchy. It is well-known in the groupoid literature \cite{brown87},\cite{weinstein} that a general group action gives rise to a transformation groupoid. More precisely,

\begin{definition}\cite{morton} Let $\phi : G \ra Aut(X)$, be a usual (set-theoretic) group action then transformation groupoid $X \wquot G$ is the groupoid consisting of:
	\begin{itemize}
		\item \textbf{Objects}: each element $x \in X $; denoted as $\mbox{Ob}(X \wquot G) = X$
		\item \textbf{Morphisms}: $(X \wquot G)(x,y) = (g,x) \in G \times X$, with $\phi_g(x)=y$
		\item \textbf{Composition}: $(g', \phi_g(x))\circ (g, x)= (g'g,x)$
	\end{itemize}
	\label{def:transfngroupoid}
\end{definition}

Now we show that every transformation groupoid is also a base structured category.  

\begin{proposition}
	A classic transformation groupoid $X \wquot G$ is simply a base-structured category ${\int_{\mathbf{G}} \bar{{F}}}$. 	
\end{proposition}
\begin{proof}
	The group $G$ can be viewed as a single object category $\mathbf{G}$ with an object $\star$ and $\mathbf{G}(\star,\star) = G$. Now consider a strict contravariant functor $\bar{F}: \mathbf{G} \rightarrow \mathbf{Set}$ between small categories, then $\int_{\mathbf{G}} \bar{{F}}$ is a category with objects the pairs $(\star,x)$ where $\star \in \mbox{Ob}(\mathbf{G})$ and $x \in \bar{F}(\star)=X$, morphisms the pairs $(g,x): (\star,x) \rightarrow (\star,y)$ where $g:\star \rightarrow \star \in \mathbf{G}$, $x = \bar{F}g(y)$. The composition is given by $(g',{y}) \bullet (g,{x}) = (g' \circ g,{x})$ since $x = \bar{{F}}g\bar{{F}}g'(z)$. Now comparing with the classic $X \wquot G$, we find that the objects of $X \wquot G$ are simply the objects of ${\int_{\mathbf{G}} \bar{{F}}}$ with the first component dropped which is simply a relabeling with the morphisms staying the same.
\end{proof} 

With this result at our disposal we can now illustrate the well-known difference in the modeling of local and global symmetries through choice of different base categories.

\subsection{Global symmetry of a Homogeneous structure using a single object base category} 
\label{bscgroup}
A homogeneous structure of any sort could be modeled as an object of a general category $\mathbf{D}$ (abstract or concrete). Then a single object base category and the corresponding base structured categories serve to capture the essence of global structure or symmetry on the homogeneous object. More precisely, let $\mathbf{F}: \mathbf{C} \xrightarrow{F} \mathbf{D} \xrightarrow{I} \mathbf{Cat}$ be the functor with $\mathbf{C}$ being a single object category with isomorphisms or more precisely an abstract group (or equivalently a pointed and connected groupoid). In this case, we have the image category $F(\mathbf{C})$ containing a single object $X$ which is precisely the homogeneous structure whose symmetry we are modeling. In the terminology of fibration there a single fibre $\mathbf{E}_{\star}$ within the total category. This necessarily implies that the cartesian lifts of a base arrow are nothing but restrictions of morphisms such as $(f,Ff)$ on the whole underlying set $X$ when $\mathbf{D}$ is $\mathbf{Set}$ or any concrete category. Let us illustrate the idea so far using an example of a cyclic group of order 3, $G_3$ acting on a simple finite set $X = \{1,2,3\}$ using visual techniques of \cite{VGT}. Here $\mathbf{D}$ is chosen as $\mathbf{FinSet}$. We have $\mathbf{F}: \mathbf{G_3} \xrightarrow{F} \mathbf{Finset} \xrightarrow{I} \mathbf{Cat}$ and we consider abstract action as defined in \cite{salilp1} of this sequel. Note that this abstract action is practically expressed as a concrete action $\mathbb{F}: \mathbf{G_3} \xrightarrow{F} \mathbf{Finset} \xrightarrow{U} \mathbf{Set}$ using the standard concrete base-structured category $\mathcal{X} \rtimes_{\mathbb{F}} \mathbf{G_3}$  again defined in \cite{salilp1} of this sequel which is nothing but essentially same as the classic transformation groupoid as we have seen earlier. However we will form base structured category of only abstract action which implies not decomposing the object $X$ into (underlying) elements to emphasize the fact that it is considered as one single global object or precisely \textbf{single basepoint (say $\star$) corresponding to a global homogeneous structure $\mathcal{X}$}.

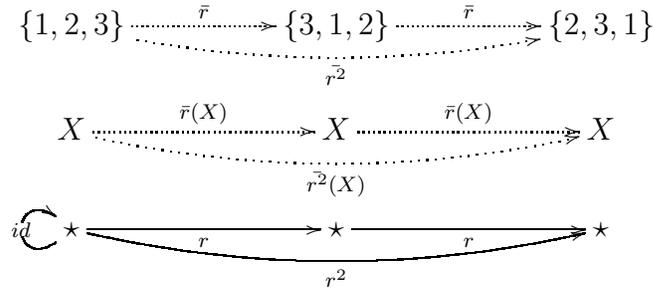
\begin{figure}
	\begin{equation*}
	\xymatrix{
		\{1,2,3\} \ar@{.>}@/_1pc/[rrrr]_{\bar{r^2}} \ar@{.>}[rr]^{\bar{r}}  &  & \{3,1,2\} \ar@{.>}[rr]^{\bar{r}} & & \{2,3,1\} \\
		X \ar@{.>}@/_1pc/[rrrr]_{\bar{r^2}(X)} \ar@{.>}[rr]^{\bar{r}(X)}  &  & X \ar@{.>}[rr]^{\bar{r}(X)} & & X \\
		\star \ar@(dl,ul)[]|{id} \ar@/_1pc/[rrrr]_{r^2} \ar[rr]_{r} &  & \star \ar[rr]_{r} & & \star 
	}
	\end{equation*}
	\caption{symmetry of a set using group as a base category.}
	\label{fig:symm1}
\end{figure}

Here $\mathcal{X} = \{1,2,3\}$ while the category $\mathcal{X} \rtimes_{\mathbf{F}} \mathbf{G_3}$ consists of:
\begin{itemize}
	\item \textbf{objects}: a single object $(\star,X)$ denoted by $\mbox{Ob}(\mathcal{X} \rtimes_{\mathbf{F}} \mathbf{G_3})$
	
	\item \textbf{morphisms}: the collection $\int_{\mathbf{G_3}} {F}((\star,X),(\star,X))=\{(r,\mathrm{id}_{X}):(\star,X)\rightarrow (\star,X)\}$
	
	\item \textbf{identity}: for the single $(\star,X)$, the morphism $\mathrm{id}_{(\star,X)} = (\mathrm{id}_{\star},\mathrm{id}_{X})$
	
	\item \textbf{composition}: if $(r,r^2)\mapsto r\circ r^2$ in $\mathbf{G_3}$ then $((r,\mathrm{id}_{X}),(r^2,\mathrm{id}_{X}))\mapsto (r,\mathrm{id}_{X})\circ (r^2,\mathrm{id}_{X})=(r\circ r^2,\mathrm{id}_{X}\cdot \mathrm{id}_{X})$
	
	\item \textbf{unit laws}: for $(r,\mathrm{id}_{X})$, 
	$(\mathrm{id}_{\star},\mathrm{id}_{X}) \bullet (r,\mathrm{id}_{X}) = (r,\mathrm{id}_{X}) = (r,\mathrm{id}_{X}) \bullet (\mathrm{id}_{\star},\mathrm{id}_{X})$
	
	\item \textbf{associativity}: $(r^2,\mathrm{id}_{X})\bullet ((r,\mathrm{id}_{X})\bullet (r,\mathrm{id}_{X}))=((r^2,\mathrm{id}_{X})\bullet (r,\mathrm{id}_{X}))\bullet (r,\mathrm{id}_{X}) =(r^2,\mathrm{id}_{X})\bullet (r,\mathrm{id}_{X})\bullet (r,\mathrm{id}_{X})$
\end{itemize}

One last time we wish to remind the reader that although this category appears to have trivial arrows in its second component there is truly underlying action of the functors such as $\mathbf{F}r$ on the category $\mathbf{F} (\star) = X$ as explained in this sequel. In case of confusion, one might choose to view this category as $(F,\mathbf{C},\mathbf{D})$. Thus a single object multi-arrow base category introduces a specific (global) symmetry structure on $X$ as shown in Figure~\ref{fig:symm1} expressed by $\mathcal{X} \rtimes_{\mathbf{F}} \mathbf{G_3}$ which intuitively signifies the fact that the structure of $X$ is expressed relative to a chosen base.

\subsection{Local and global symmetries of a non-homogeneous structure using a multi-object base category}
\label{bscgroupoid}
A non-homogeneous structure of any sort could be modeled as a collection of homogeneous components or multiple objects of a general category $\mathbf{D}$ (abstract or concrete). Then a multiple objects base category and the corresponding base structured categories serve to capture the essence of local and global structures or symmetries on the complete object. More precisely, let $\mathbf{F}: \mathbf{C} \xrightarrow{F} \mathbf{D} \xrightarrow{I} \mathbf{Cat}$. be the functor with $\mathbf{C}$ being a multi-object category with isomorphisms or more precisely a non-pointed arbitrarily connected groupoid. In this case, we have the image category $F(\mathbf{C})$ containing multiple objects $X_1,X_2,X_3 ...$ or $X = X_1 \amalg X_2 \amalg X_3 \amalg ...$ which is precisely the complete non-homogeneous structure whose symmetry we are modeling. In the terminology of fibration there are multiple fibres of the type $\mathbf{E}_{\star_1}$ in the total base structured category. This implies that cartesian lifts of a base arrow are local transformations only defined on these local objects or spaces.

Let us again illustrate this idea using an example of a groupoid $\mathcal{G}_3$ acting on a simple finite set $X = \{1,2,3\}$. We have $\mathbf{F}: \mathcal{G}_3 \xrightarrow{F} \mathbf{Finset} \xrightarrow{I} \mathbf{Cat}$ and we consider abstract action again. Here the emphasis is on the fact that we decompose the object $X$ into (underlying) elements or $X = x_1 \amalg x_2 \amalg x_3$. Thus same symmetry structure on $X$ can be expressed by changing the base-structured category to a groupoid with multiple objects where the morphisms are same as those of a group. Hence $X$ automatically gets partitioned by this base providing an illustration of how same symmetry is expressed by choosing a multi object base. In essence we have \textbf{multiple basepoints (say $\star_1$,$\star_2$,...) corresponding to a disjoint union of locally homogeneous structures $\mathcal{X} = x_1 \amalg x_2 \amalg ...$}.  

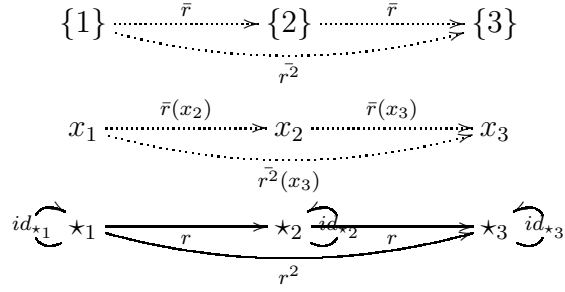
\begin{figure}
	\begin{equation*}
	\xymatrix{
		\{1\} \ar@{.>}@/_1pc/[rrrr]_{\bar{r^2}} \ar@{.>}[rr]^{\bar{r}}  &  & \{2\} \ar@{.>}[rr]^{\bar{r}} & & \{3\} \\
		x_1 \ar@{.>}@/_1pc/[rrrr]_{\bar{r^2}(x_3)} \ar@{.>}[rr]^{\bar{r}(x_2)}  &  & x_2 \ar@{.>}[rr]^{\bar{r}(x_3)} & & x_3 \\
		\star_1 \ar@(dl,ul)[]|{id_{\star_1}} \ar@/_1pc/[rrrr]_{r^2} \ar[rr]_{r} &  & \star_2 \ar@(dr,ur)[]|{id_{\star_2}} \ar[rr]_{r} & & \star_3 \ar@(dr,ur)[]|{id_{\star_3}} 
	}
	\end{equation*}
	\caption{symmetry of a set using groupoid as a base category.}
	\label{fig:symm2}
\end{figure}

Here $\mathbf{B} = \mathcal{G}_3$, a groupoid with three distinct objects as shown in the base while the category $\mathcal{X} \rtimes_{\mathbf{F}} \mathcal{G}_3$ consists of:,

\begin{itemize}
	\item \textbf{objects} $(\star_1,x_1),(\star_2,x_2),(\star_3,x_3)$ denoted by $\mbox{Ob}(\mathcal{X} \rtimes_{\mathbf{F}} \mathcal{G}_3)$
	
	\item \textbf{morphisms} the collection ${\mathcal{X} \rtimes_{\mathbf{F}} \mathcal{G}_3} ((\star_1,x_1),(\star_2,x_2))=\{(r,\mathrm{id}_{x_2}):(\star_1,x_1)\rightarrow (\star_2,x_2)\}$
	
	\item \textbf{identity}: for each object $(\star_1,x_1)$, the morphism $\mathrm{id}_{(\star_1,x_1)} = (\mathrm{id}_{\star_1},\mathrm{id}_{x_1})$
	
	\item \textbf{composition}: $(r,\mathrm{id}_{x_3}) \bullet (r,\mathrm{id}_{x_2}) = (r \circ r,\mathrm{id}_{x_3} \cdot \mathbf{F}r(\mathrm{id}_{x_2})) = (r^2,\mathrm{id}_{x_3})$
	
	\item \textbf{unit laws}: for $(r,\mathrm{id}_{x_2})$, 
	$(\mathrm{id}_{\star_2},\mathrm{id}_{x_2}) \bullet (r,\mathrm{id}_{x_2}) = (r,\mathrm{id}_{x_2}) = (r,\mathrm{id}_{x_2}) \bullet (\mathrm{id}_{\star_1},\mathrm{id}_{x_1})$    
	
	\item \textbf{associativity}:  $(r^2,\mathrm{id}_{x_1})\bullet ((r,\mathrm{id}_{x_3})\bullet (r,\mathrm{id}_{x_2}))=((r^2,\mathrm{id}_{x_1})\bullet (r,\mathrm{id}_{x_3}))\bullet (r,\mathrm{id}_{x_2}) =(r^2,\mathrm{id}_{x_1})\bullet (r,\mathrm{id}_{x_3})\bullet (r,\mathrm{id}_{x_2})$
\end{itemize}

Thus a multi object connected base now models same symmetry but from within the object since $X$ is no longer treated as single relative to the base but as co-product or disjoint union of three distinct singleton objects $X = x_1 \amalg x_2 \amalg x_3$ as shown in Figure~\ref{fig:symm2}. Hence given set $X$ when treated as a single object requires a group (or equivalently a single object category with invertible arrows) to express its symmetry or structure. The same set $X$ when treated as multi-object (commonly termed as object of objects) i.e $X = x_1 \amalg x_2 \amalg x_3$ requires a multi-object groupoid as a base category to express the same symmetry.

\section{Hierarchy of symmetry I: Leyton's generative theory}
\label{sec:wreath}
In this section we describe the difference in the modeling of generative theory of shape proposed in \cite{Leyton86a},\cite{Leyton86b},\cite{Leyton86c},\cite{Leyton01} using groups, groupoids and base structured categories. The model using hierarchy of groups leading to wreath groups is utilized by Leyton in proposing the generative theory of shape; see \cite{Leyton01} and references therein. However the same theory is could be generally modeled using hierarchy of groupoids following \cite{weinstein}.

Consider a simple example of a non-homogeneous structure $X = X_1 \cup X_2$ as shown in Figure~\ref{fig:nonhom} where we have
disjoint union of two homogeneous structures $X_1$ and $X_2$. The complete symmetry of such a structure is well studied in the literature \cite{weinstein} and captured by groupoid. What we wish to illustrate in this section is that the symmetry of such a structure was shown to be also captured by a wreath group by the author in \cite{Leyton01}. We first study how this modeled whenever possibile and then make a comparative analysis with the standard groupoid model to better understand the difference utilizing the perspective offered by the base structured categories. 

\begin{figure}
	\begin{center}	
		\begin{tikzpicture}[>=latex]

		\node at (3,3) {$\mathbb{R}^2$};
		\node at (0.35,0) {$O$};
		
		\draw [-] (-5,0.35)-- (5,0.35);
		\draw [-] (0,3)-- (0,-2);
		
		\fill (-4,0) circle (0.05cm);
		\fill (-2,0) circle (0.05cm);
		\fill (-3,1) circle (0.05cm);
		
		\node at (-4,-0.2) {$x_3$};
		\node at (-2,-0.2) {$x_2$};
		\node at (-3,1.2) {$x_1$};
		
		%

		\fill (4,0) circle (0.05cm);
		\fill (2,0) circle (0.05cm);
		\fill (3,1) circle (0.05cm);
		
		\node at (4,-0.2) {$x_5$};
		\node at (2,-0.2) {$x_6$};
		\node at (3,1.2) {$x_4$};
		
		
		\end{tikzpicture}
	\end{center}
	\caption{Non Homogeneous Structure $X = X_1 \cup X_2$ illustrating 2-level symmetry.}
	\label{fig:nonhom}
\end{figure}
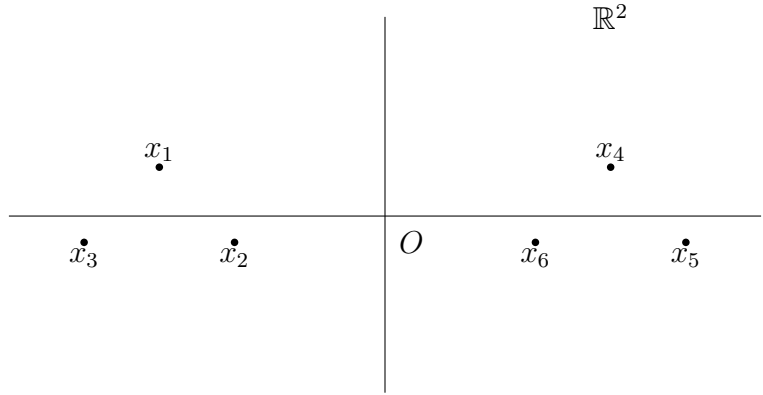

In \cite{Leyton01} the whole non-homogeneous structure $X$ (generated by an artist) is considered to be the \textbf{transfer} of homogeneous prototypes such as $X_1$ in the example we have considered. The homogeneous prototypes are modeled independently in their own separate global spaces. Thus in Figure~\ref{fig:local} the symmetry of prototype $X_1$ is captured in Euclidean space $\mathbb{R}^2$ by the familiar group $D_3$ (which truly acts on the whole space $\mathbb{R}^2$ or in other words a subgroup of $O(2)$).  

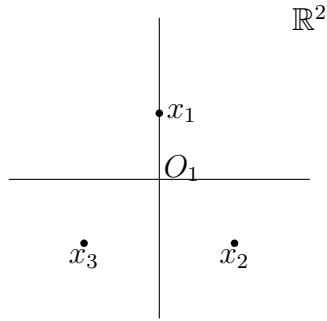
\begin{figure}
	\begin{center}	
		\begin{tikzpicture}[>=latex]

		\draw [-] (2,-0.15)-- (-2,-0.15);
		\draw [-] (0,2) -- (0,-2);
		
		\fill (0,0.73) circle (0.05cm);
		\fill (-1,-1) circle (0.05cm);
		\fill (1,-1) circle (0.05cm);
		
		\node at (2,2) {$\mathbb{R}^2$};
		\node at (0.3,0) {$O_1$};
		
		\node at (-1,-1.2) {$x_3$};
		\node at (1,-1.2) {$x_2$};
		\node at (0.3,0.73) {$x_1$};
		
		\end{tikzpicture}
	\end{center}
	\caption{symmetry of prototype $X_1$ captured by $D_3$.}
	\label{fig:local}
\end{figure}

Hence two local symmetries are captured by dihedral groups of order three. The first local symmetry is captured by the permutation action of group $D_3$ on $X_1$ where $X_1 = \{1,2,3\} = \{x_1,x_2,x_3\}$. Similarly the second local symmetry is captured by the permutation action of another group $D_3$ on $X_2$ where $X_1 = \{4,5,6\} = \{x_4,x_5,x_6\}$. Thus we have two functors which denote actions of type $\mathbf{F}: \mathbf{D_3} \xrightarrow{F} \mathbf{Set} \xrightarrow{I} \mathbf{Cat}$. Now since a group when considered as one object category (refer Section~\ref{sec:pointedness} for single object species of structures) always determines a symmetry of global nature, the morphisms determined by its elements are well defined on every global point of the object $X$. For a collection of multiple independent objects $X = X_1 \cup X_2$  with individual symmetries, the global symmetry of this collection becomes the product of individual symmetries, since the collection gets treated as a single object ${X}$ with no specific partition or injections of sub objects giving rise to a product category $\mathbf{D_3} \times \mathbf{D_3} = \mathbf{B}_2$ or a single functor $\mathbf{F_2}: \mathbf{D_3} \xrightarrow{F} \mathbf{Set} \xrightarrow{I} \mathbf{Cat}$ or as a base category $X \rtimes_{\mathbf{F_2}} (\mathbf{D_3} \times \mathbf{D_3})$. Here since all the objects fuse into a single object (correspondingly captured by single object of base category $\mathbf{B}_2$), there remains a single base point relative perspective. Finally the global symmetry is captured by the action of a reflection group (reflection across y-axis of Euclidean plane) of order two which we will set-wise as $Z_2 = \{e,t\}$. The notation $X_1 \cup X_2 = X_1 \times \mbox{Ob}(\mathbf{Z}_2)$ signifies the fact that complete object is a Leyton transfer of the 3 element set $X_1$ as a prototype. 

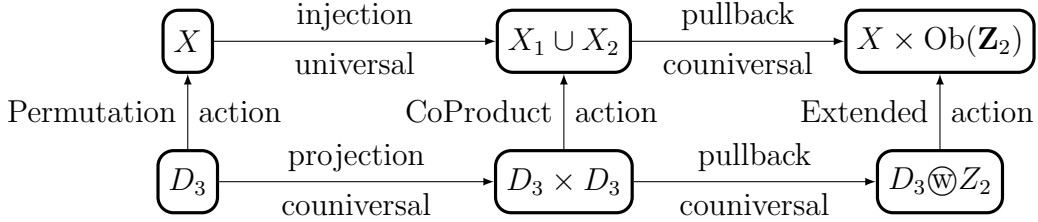
\begin{figure}
	\begin{tikzpicture}[>=latex]
	\tikzstyle{ibox}=[draw=black,very thick,shape=rectangle,rounded corners=0.5em,inner sep=4pt, minimum height=2em,text badly centered,fill=white!5!white];
	
	\node[ibox,align=left] (x) at (0,3) {$X_1 \cup X_2$};  
	\node[ibox,align=left] (x1) at (-5,3) {$X$}; 
	\node[ibox,align=left] (x2) at (5,3) {$X \times \mbox{Ob}(\mathbf{Z}_2)$}; 
	
	\node[ibox,align=left] (y) [below=of x] {$D_3 \times D_3$};
	\node[ibox,align=left] (y1) [below=of x1] {$D_3$}; 
	\node[ibox,align=left] (y2) [below=of x2] {$D_3 \textcircled{w} Z_2$}; 
	
	\draw [->] (x1) -- (x) node [pos=0.5,above] {injection} node [pos=0.5,below] {universal};
	\draw [->] (x) -- (x2) node [pos=0.5,above] {pullback} node [pos=0.5,below,align=left] {couniversal};
	
	\draw [->] (y1) -- (y) node [pos=0.5,above] {projection} node [pos=0.5,below] {couniversal};
	\draw [->] (y) -- (y2) node [pos=0.5,above] {pullback} node [pos=0.5,below,align=left] {couniversal};
	
	\draw [->] (y1) -- (x1) node [pos=0.5,left] {Permutation} node [pos=0.5,right] {action};
	\draw [->] (y) -- (x) node [pos=0.5,left] {CoProduct} node [pos=0.5,right] {action};
	\draw [->] (y2) -- (x2) node [pos=0.5,left] {Extended} node [pos=0.5,right] {action};
	
	\end{tikzpicture}
	\caption{Leyton's generative model using groups}
	\label{fig:leyton1} 
\end{figure}

To summarize, here on a single object $X =\{1,2,3,4,5,6\}$ we have first action of $\mathbf{D_3} \times \mathbf{D_3}$ or the coproduct action in the group theory, which captures the local structure or symmetries but from a global perspective where the given set $Y$ is considered as one object $\star$. The next level of structure or symmetry between the local symmetries is captured by action of $\mathbf{Z_2}$ on $\mathbf{D_3} \times \mathbf{D_3}$ by automorphisms, leading to a wreath group. The 2-hierarchy of symmetries is therefore captured by a group viewed as a single object category $\mathbf{C_3}$ \textcircled{w} $\mathbf{C_2}$ which is a special case of semi-direct group $(\mathbf{C_3} \times \mathbf{C_3}) \rtimes \mathbf{C_2}$. This is the restriction to category of groups or global symmetries.

Now we capture the symmetry of the same non-homogeneous structure using the widely known and studied groupoid model following \cite{weinstein}

Here let the whole object $X$ is modeled as contained in the single space $\mathbb{R}^2$. As we have seen that every group action results into a transformation groupoid let us denote the groupoid formed by the action of group $\Gamma$ of rigid motions of $\mathbb{R}^2$ on $\mathbb{R}^2$ as $G(\Gamma,\mathbb{R}^2) = \{((x,y),\gamma) | x \in \mathbb{R}^2, y \in \mathbb{R}^2, \gamma \in \Gamma, x= \gamma y \} $. Then the groupoid leaving the $X$ invariant is just $G(\Gamma,\mathbb{R}^2)|_X$ consisting of those $g \in G(\Gamma,\mathbb{R}^2)$ where the domain and codomains of the $g$ are simply restricted to $X$. This is the groupoid capturing the reflection symmetry and is isomorphic to the action groupoid produced from action of reflection group $Z_2 = \{e,t\}$ on the set $X$. The local symmetry is captured by another groupoid $G_{loc}$. To define this groupoid we consider the plane $\mathbb{R}^2$ as the disjoint union of $X_1 = X \cap X_1$, $X_2 = X \cap X_2$ and $\mathbb{R}^2 \setminus X$ (plane puncturing out the six points). More precisely  $G_{loc} = \{((x,y),\gamma) | x \in X, y \in X, \gamma \in \Gamma, x= \gamma y \} $ for which $y$ has a neighborhood $u$ in $\mathbb{R}^2$ such that $\gamma(u \cap X_i) \subseteq X_i$. This groupoid captures the dihedral local symmetries of $X_1$ and $X_2$. To connect with the generative theory of Leyton; the local and global groupoids could viewed as base (groupoid) structured categories as shown in Figure~\ref{fig:symm2}.

A groupoid being a multiple object category (refer Section~\ref{sec:pointedness} for multi-object species of structures) captures a symmetry of a local nature since the morphisms determined by its elements are locally defined on specific local point of the object $X$. Thus a collection of multiple independent objects $\mathcal{X} = X_1 \cup X_2$  with individual symmetries, the global symmetry of this collection becomes the amalgamation of individual local symmetries. The collection gets naturally treated as a collection of multiple objects $\mathcal{X}$ with partition or injections of subobjects $X_1,X_2$ giving rise to a coproduct category $\mathbf{B}_2$ which is the Coproduct Groupoid $\mathcal{D}_3 \amalg \mathcal{D}_3$  

Intuitively the separation of objects $\{1,2,3\}$ and $\{4,5,6\}$ (correspondingly captured by multiple objects of base category $\mathbf{B}_2$) keeps multi-basepoints relative perspective intact in this structure allowing the possibility of adding arrows across different objects locally. Indeed every individual object in the whole co-product object retains its individuality and the domain and range set remains divided into multiple subsets on which arrows get defined.

In summary, here the given set $X$ is no longer treated as a single object but a hierarchy of multi-objects leading to global, local, sub-local perspective on symmetries of $X$ (the example we are considering will result into two levels of Base corresponding to two hierarchy levels of symmetries or structures). Roughly speaking this is akin to horizontal categorification where we break a single global object $X$ into multiple sub-objects $X_1$ and $X_2$ viewed as fibred on $\mathbf{B}_1 = \mathcal{Z}_2$ with two objects $\ast_1$ and $\ast_2$. The fibres containing single objects $X_1$ and $X_2$ respectively are further broken into second level of sub-objects $\{1\}$,$\{2\}$,$\{3\}$ and $\{4\}$,$\{5\}$,$\{6\}$ viewed as fibred on $\mathcal{D}_3$ and $\mathcal{D}_3$ with three objects in each of the base groupoids.

The notation $X_1 \cup X_2 = X \times \mbox{Ob}(\mathcal{Z}_2)$ signifies the fact that complete object is an Ehresmann transport (see \cite{brown87} and \cite{charles}) of the 3 element set $X_1$ as a prototype. 

\begin{figure}[h]
	
	\begin{tikzpicture}[>=latex]
	
	\tikzstyle{ibox}=[draw=black,very thick,shape=rectangle,rounded corners=0.5em,inner sep=4pt, minimum height=2em,text badly centered,fill=white!5!white];
	
	\node[ibox,align=left] (x) at (0,3) {$X_1 \cup X_2$};  
	\node[ibox,align=left] (x1) at (-5,3) {$X$}; 
	\node[ibox,align=left] (x2) at (5,3) {$X \times \mbox{Ob}(\mathcal{Z}_2)$}; 
	
	\node[ibox,align=left] (y) [below=of x] {$(\mathcal{D}_3 \amalg \mathcal{D}_3)=\mathbf{B}_2$};
	\node[ibox,align=left] (y1) [below=of x1] {$\mathcal{D}_3$}; 
	\node[ibox,align=left] (y2) [below=of x2] {$\mathbf{B}_2 \rtimes \mathcal{Z}_2$}; 
	
	\draw [->] (x1) -- (x) node [pos=0.5,above] {injection} node [pos=0.5,below] {universal};
	\draw [->] (x) -- (x2) node [pos=0.5,above] {pullback} node [pos=0.5,below,align=left] {couniversal};
	
	\draw [->] (y1) -- (y) node [pos=0.5,above] {injection} node [pos=0.5,below] {universal};
	\draw [->] (y) -- (y2) node [pos=0.5,above] {pullback} node [pos=0.5,below,align=left] {couniversal};
	
	\draw [->] (y1) -- (x1) node [pos=0.5,left] {Permutation} node [pos=0.5,right] {action};
	\draw [->] (y) -- (x) node [pos=0.5,left] {CoProduct} node [pos=0.5,right] {action};
	\draw [->] (y2) -- (x2) node [pos=0.5,left] {Extended} node [pos=0.5,right] {action};
	
	\end{tikzpicture}
	
	\caption{Leyton's generative model using groupoids}
	\label{fig:leyton2} 
\end{figure}
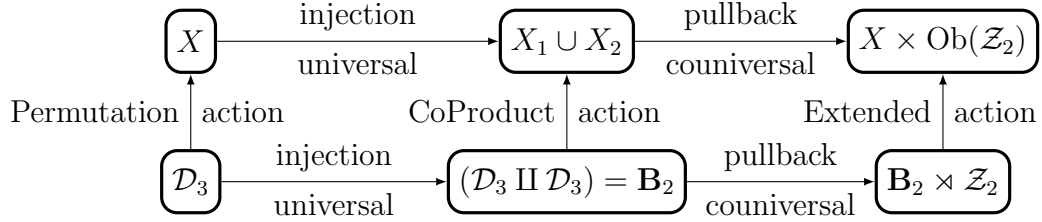

Noting that actions of categories give rise to base structured categories, we get an equivalent model as shown in Figure~\ref{fig:leyton3} using base structured categories.

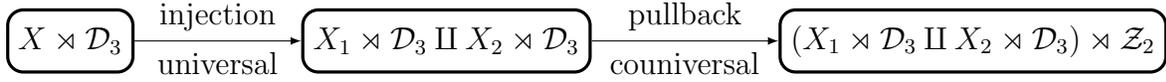
\begin{figure}[h]
	
	\begin{tikzpicture}[>=latex]
	
	\tikzstyle{ibox}=[draw=black,very thick,shape=rectangle,rounded corners=0.5em,inner sep=4pt, minimum height=2em,text badly centered,fill=white!5!white];
	
	\node[ibox,align=left] (x) at (0,3) {$X_1 \rtimes \mathcal{D}_3 \amalg X_2 \rtimes \mathcal{D}_3$};  
	\node[ibox,align=left] (x1) at (-5,3) {$X \rtimes \mathcal{D}_3$}; 
	\node[ibox,align=left] (x2) at (7,3) {$(X_1 \rtimes \mathcal{D}_3 \amalg X_2 \rtimes \mathcal{D}_3) \rtimes \mathcal{Z}_2$}; 
	
	\draw [->] (x1) -- (x) node [pos=0.5,above] {injection} node [pos=0.5,below] {universal};
	\draw [->] (x) -- (x2) node [pos=0.5,above] {pullback} node [pos=0.5,below,align=left] {couniversal};
	
	\end{tikzpicture}
	
	\caption{Leyton's generative model using base structured categories}
	\label{fig:leyton3} 
\end{figure}

Thus, natural generativity and transfer could be more faithfully modeled replacing category theoretic analogues of the sets and groups giving rise to base structured categories as shown in Figure~\ref{fig:generative} which makes use of the perspective of functor as a category action. In \cite{hendrickx} the authors have raised valid concerns on the idea of successive referencing. It appears to us at this moment that the issue might be tackled using the notion of base structured categories along with hierarchy of bases leading to n-categories; refer Section~\ref{h} and Section~\ref{2g} for the details.  

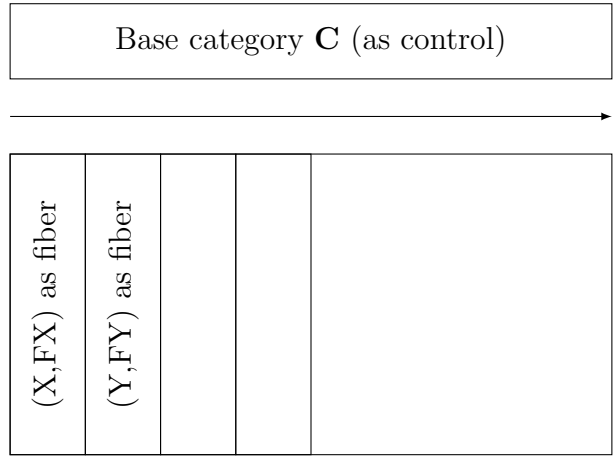
\begin{figure}
	\begin{center}
		\begin{tikzpicture}[>=latex]
		
		\draw (-4,1) rectangle (4,2) node [pos=0.5] {Base category $\mathbf{C}$ (as control)}; 
		
		\draw (-4,-4) rectangle (4,0);
		
		\draw (-4,-4) rectangle (-3,0) node [pos=0.5,rotate=90] {(X,FX) as fiber};
		\draw (-3,-4) rectangle (-2,0) node [pos=0.5,rotate=90] {(Y,FY) as fiber};
		\draw (-2,-4) rectangle (-1,0);
		\draw (-1,-4) rectangle (0,0);
		
		\draw[->,thin] (-4,0.5) -- (4,0.5);
		
		\end{tikzpicture}
		
		\caption{Base structured categorical $\mathcal{X} \rtimes_{\mathbf{F}} \mathbf{C}$ model of Leyton's generative theory}
		\label{fig:generative}
	\end{center}
\end{figure}

\section{Hierarchy of symmetry II: Base structured categorical Models}
\label{h}

In continuation of last section, here we revisit the modeling of arbitrary multiple levels of symmetry (widely known as local/global or internal/external) algebraically using groups, groupoids and base structured categories.

\subsection{Hierarchy of Symmetries}
For considering the hierarchy of bases we utilize notations with subscripts $1,2,....n$ to denote a level in hierarchy. Thus $\mathbf{B}_1$ will denote a base category at the first (outermost) level whereas $\mathcal{X}_{1}$ denotes coproduct object at first level. The subscripts $a,b,...$ will be used to denote the objects or base categories at a given level. Thus $\mathcal{X}_{1a}$ and $\mathcal{X}_{1b}$ will denote two objects within the coproduct object $\mathcal{X}_{1}$ whereas $\mathbf{B}_{2a}$, $\mathbf{B}_{2b}$ denotes base categories at second level (inner) level.

Hence a 2-level fibration of objects will be denoted by $[(\mathcal{X}_{2a} \rtimes_{\mathbf{F_{2a}}} \mathbf{B}_{2a}) \amalg (\mathcal{X}_{2b} \rtimes_{\mathbf{F_{2b}}} \mathbf{B}_{2b}) \amalg ...] \rtimes_{\mathbf{F_{1}}} \mathbf{B}_1$ where $\mathcal{X}_{1a} = (\mathcal{X}_{2a} \rtimes_{\mathbf{F_{2a}}} \mathbf{B}_{2a})$, $\mathcal{X}_{1b} = (\mathcal{X}_{2b} \rtimes_{\mathbf{F_{2b}}} \mathbf{B}_{2b})$, $(\mathcal{X}_{1a} \amalg \mathcal{X}_{1b} \amalg ...) = \mathcal{X}_{1}$ and $\mathbf{B}_{2} = (\mathbf{B}_{2a} \amalg \mathbf{B}_{2b} \amalg ...)$.       

Again taking the same example of Section~\ref{sec:wreath} where we have independent copies of $X$ say $X_1$ and $X_2$  which are translates of each other forming a total set $Y = X_1 \amalg X_2$, the complete symmetry or structure of $Y$ can be expressed or formulated again using groups, groupoids and corresponding base structured categories.

\subsection{Using Hierarchy of Groups $[(\mathcal{X} \rtimes_{\mathbf{F_{2}}} (\mathbf{G_3} \times \mathbf{G_3})] \rtimes_{\mathbf{F_{1}}} \mathbf{G_2}$}

\begin{figure}
	\begin{equation*}
	\xymatrix{
		\{1,2,3,4,5,6\} \ar@{.>}@/_1pc/[rrrr]_{\overline{r^2 \times r^2}(Y)} \ar@{.>}[rr]^{\overline{r \times r}(Y)}  &  & \{3,1,2,6,4,5\} \ar@{.>}[rr]^{\overline{r \times r}(Y)} & & \{2,3,1,5,6,4\} \\
		Y \ar@{.>}@/_1pc/[rrrr]_{\overline{r^2 \times r^2}(Y)} \ar@{.>}[rr]^{\overline{r \times r}(Y)}  &  & Y \ar@{.>}[rr]^{\overline{r \times r}(Y)} & & Y \\
		\star \ar@(dl,ul)[]|{id_{\star}} \ar@/_1pc/[rrrr]_{r^2 \times r^2} \ar[rr]_{r \times r} &  & \star \ar[rr]_{r \times r} & & \star 
	}
	\end{equation*}
	\caption{Inner symmetry using groups}
	\label{fig:hgroup1}
\end{figure}
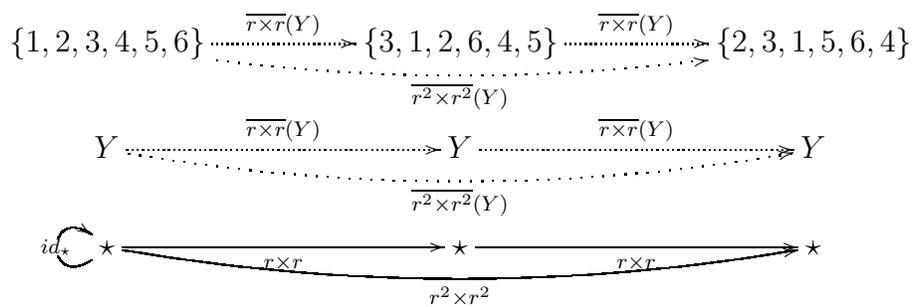

Here second level (inner) base $\mathbf{B}_2 = \mathbf{G_3} \times \mathbf{G_3}$, the direct product group is a single object category. Here $\mathbf{F_2}: \mathbf{B}_2 \xrightarrow{F_2} \mathbf{FinSet} \xrightarrow{I} \mathbf{Cat}$ with $ F_2(\star)= Y = X_1 \cup X_2 = \{1,2,3,4,5,6\}$ and ${Y} \rtimes_{\mathbf{F_{2}}} \mathbf{G_3} \times \mathbf{G_3}$ consisting of:
\begin{itemize}
	\item \textbf{objects}: a single object $(\star,Y)$ denoted by $\mbox{Ob}({Y} \rtimes_{\mathbf{F_{2}}} \mathbf{G_3} \times \mathbf{G_3})$
	
	\item \textbf{morphisms}: the collection $({Y} \rtimes_{\mathbf{F_{2}}} \mathbf{G_3} \times \mathbf{G_3})((\star,Y),(\star,Y))=\{(r \times r,\mathrm{id}_{Y}):(\star,Y)\rightarrow (\star,Y)\}$
	
	\item \textbf{identity}: for the single $(\star,Y)$, the morphism $\mathrm{id}_{(\star,Y)} = (\mathrm{id}_{\star},\mathrm{id}_{Y})$
	
	\item \textbf{composition}: if $(r \times r,r^2 \times r^2)\mapsto (r \times r) \circ (r^2 \times r^2)$ in $\mathbf{G_3}\times \mathbf{G_3}$ then $((r \times r,\mathrm{id}_{Y}),(r^2 \times r^2,\mathrm{id}_{Y}))\mapsto (r \times r,\mathrm{id}_{Y})\circ (r^2 \times r^2,\mathrm{id}_{Y})=((r \times r) \circ (r^2 \times r^2),\mathrm{id}_{Y}\cdot \mathrm{id}_{Y})$
	
	\item \textbf{unit laws}: for $(r \times r,\mathrm{id}_{Y})$, 
	$(\mathrm{id}_{\star},\mathrm{id}_{Y}) \circ (r \times r,\mathrm{id}_{Y}) = (r \times r,\mathrm{id}_{Y}) = (r \times r,\mathrm{id}_{Y}) \circ (\mathrm{id}_{\star},\mathrm{id}_{Y})$
	
	\item \textbf{associativity}: $(r^2 \times r^2,\mathrm{id}_{Y})\circ ((r \times r,\mathrm{id}_{Y})\circ (r \times r,\mathrm{id}_{Y}))=((r^2 \times r^2,\mathrm{id}_{Y})\circ (r \times r,\mathrm{id}_{Y}))\circ (r \times r,\mathrm{id}_{Y}) =(r^2 \times r^2,\mathrm{id}_{Y})\circ (r \times r,\mathrm{id}_{Y})\circ (r \times r,\mathrm{id}_{Y})$
\end{itemize}

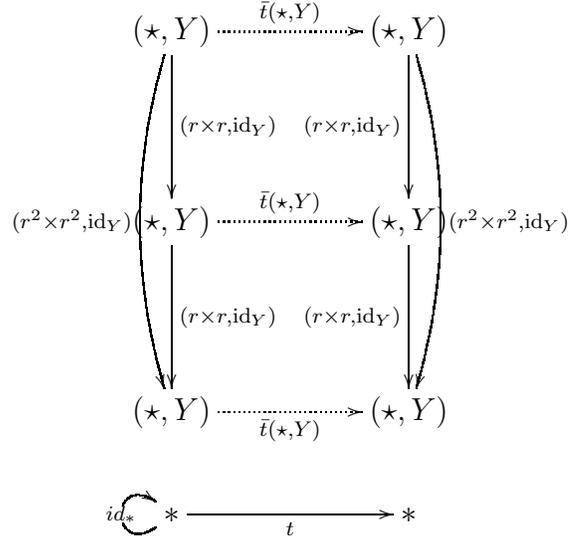
\begin{figure}
	\begin{equation*}
	\xymatrix{
		(\star,Y) \ar@/_1pc/[dddd]_{(r^2 \times r^2,\mathrm{id}_{Y})} \ar[dd]^{(r \times r,\mathrm{id}_{Y})} \ar@{.>}[rr]^{\bar{t}(\star,Y)} &  & (\star,Y) \ar[dd]_{(r \times r,\mathrm{id}_{Y})} \ar@/^1pc/[dddd]^{(r^2 \times r^2,\mathrm{id}_{Y})} \\
		& & \\
		(\star,Y) \ar[dd]^{(r \times r,\mathrm{id}_{Y})} \ar@{.>}[rr]^{\bar{t}(\star,Y)} &  & (\star,Y) \ar[dd]_{(r \times r,\mathrm{id}_{Y})}  \\
		& & \\
		(\star,Y) \ar@{.>}[rr]_{\bar{t}(\star,Y)} &  & (\star,Y) \\
		\ast \ar@(dl,ul)[]|{id_{\ast}} \ar[rr]_{t} &  & \ast 
	}
	\end{equation*}
	\caption{Outer symmetry using groups}
	\label{fig:hgroup2}
\end{figure}

Here first level (outer) base $\mathbf{B}_1 = \mathbf{G_2}$, the cyclic group of order two, is a single object category and $\mathbf{F_{1}}: \mathbf{G_2} \xrightarrow{F_1} \mathbf{Cat} \xrightarrow{I} \mathbf{Cat}$ and $\mathbf{F_{1}}(\ast) = {Y} \rtimes_{\mathbf{F_{2}}} (\mathbf{G_3} \times \mathbf{G_3})$ with $[(Y \rtimes_{\mathbf{F_{2}}} (\mathbf{G_3} \times \mathbf{G_3})] \rtimes_{\mathbf{F_{1}}} \mathbf{G_2}$ consisting of:
\begin{itemize}
	\item \textbf{objects}: a single object $(\ast,F_1(\ast))= (\ast,\star,Y)$ denoted by $\mbox{Ob}([(Y \rtimes_{\mathbf{F_{2}}} (\mathbf{G_3} \times \mathbf{G_3})] \rtimes_{\mathbf{F_{1}}} \mathbf{G_2})$
	
	\item \textbf{morphisms}: the collection of form $((\ast,\star,Y),(\ast,\star,Y))=\{(t,(r \times r,\mathrm{id}_{Y})):(\ast,\star,Y)\rightarrow (\ast,\star,Y)\}$
	
	\item \textbf{identity}: for the single $(\ast,\star,Y)$, the morphism $\mathrm{id}_{(\ast,\star,Y)} = (\mathrm{id}_{\ast},(\mathrm{id}_{\star},\mathrm{id}_{Y}))$
	
	\item \textbf{composition}: if $(t,t)\mapsto t \bullet t$ in $\mathbf{G_2}$ then $((t,(r \times r,\mathrm{id}_{Y})),(t,(r \times r,\mathrm{id}_{Y})))\mapsto (t,(r \times r,\mathrm{id}_{Y}))\bullet (t,(r \times r,\mathrm{id}_{Y}))=(t \bullet t,((r \times r) \circ (r \times r),\mathrm{id}_{Y}))$
\end{itemize}

The usual unit laws and associativity axiom can be verified easily. Henceforth we will not mention the unit laws and associativity unless explicitly required within a particular context.

\subsection{Using Hierarchy of Groupoids $[(\mathcal{X}_{2a} \rtimes_{\mathbf{F_{2a}}} \mathcal{G}_3) \amalg (\mathcal{X}_{2b} \rtimes_{\mathbf{F_{2b}}} \mathcal{G}_3)] \rtimes_{\mathbf{F_{1}}} \mathcal{G}_2$}

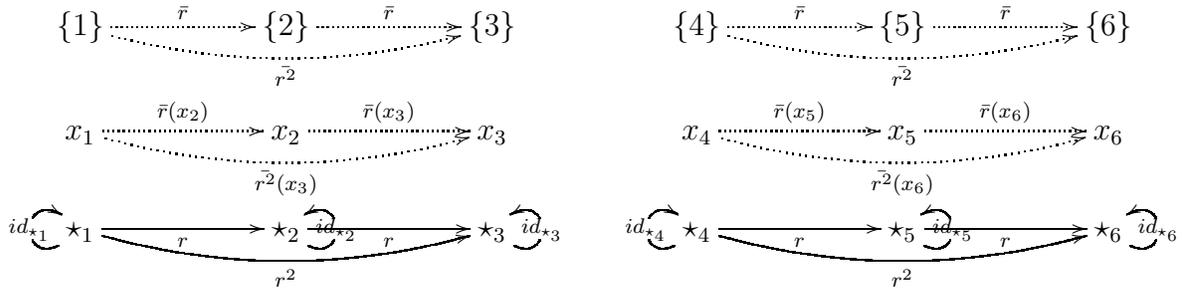
\begin{figure}
	\begin{equation*}
	\xymatrix{
		\{1\} \ar@{.>}@/_1pc/[rrrr]_{\bar{r^2}} \ar@{.>}[rr]^{\bar{r}}  &  & \{2\} \ar@{.>}[rr]^{\bar{r}} & & \{3\} \\
		x_1 \ar@{.>}@/_1pc/[rrrr]_{\bar{r^2}(x_3)} \ar@{.>}[rr]^{\bar{r}(x_2)}  &  & x_2 \ar@{.>}[rr]^{\bar{r}(x_3)} & & x_3 \\
		\star_1 \ar@(dl,ul)[]|{id_{\star_1}} \ar@/_1pc/[rrrr]_{r^2} \ar[rr]_{r} &  & \star_2 \ar@(dr,ur)[]|{id_{\star_2}} \ar[rr]_{r} & & \star_3 \ar@(dr,ur)[]|{id_{\star_3}} 
	}
	\qquad
	\xymatrix{
		\{4\} \ar@{.>}@/_1pc/[rrrr]_{\bar{r^2}} \ar@{.>}[rr]^{\bar{r}}  &  & \{5\} \ar@{.>}[rr]^{\bar{r}} & & \{6\} \\
		x_4 \ar@{.>}@/_1pc/[rrrr]_{\bar{r^2}(x_6)} \ar@{.>}[rr]^{\bar{r}(x_5)}  &  & x_5 \ar@{.>}[rr]^{\bar{r}(x_6)} & & x_6 \\
		\star_4 \ar@(dl,ul)[]|{id_{\star_4}} \ar@/_1pc/[rrrr]_{r^2} \ar[rr]_{r} &  & \star_5 \ar@(dr,ur)[]|{id_{\star_5}} \ar[rr]_{r} & & \star_6 \ar@(dr,ur)[]|{id_{\star_6}} 
	}
	\end{equation*}
	\caption{Inner symmetries using groupoids}
	\label{fig:hgroupoid1}
\end{figure}

Here second level (inner) base $\mathbf{B}_2 = \mathcal{G}_3 \amalg \mathcal{G}_3$, a groupoid with six distinct objects as shown in the base of Figure~\ref{fig:hgroupoid1}. Here $F_2: (\mathcal{G}_3 \amalg \mathcal{G}_3) \rightarrow \mathbf{FinSet}$ which is same as $F_2 = F_{2a} \amalg F_{2b}$ and ${Y} = X_1 \cup X_2$. Then ${\mathcal{X}_2} \rtimes_{\mathbf{F_{2}}} \mathbf{B}_2$ is a coproduct of two base structured categories which is $(\mathcal{X}_{2a} \rtimes_{\mathbf{F_{2a}}} \mathcal{G}_3) \amalg (\mathcal{X}_{2b} \rtimes_{\mathbf{F_{2b}}} \mathcal{G}_3)$ consisting of:
\begin{itemize}
	\item \textbf{objects} $(\star_1,x_1),(\star_2,x_2),(\star_3,x_3),(\star_4,x_4),(\star_5,x_5),(\star_6,x_6)$
	
	\item \textbf{morphisms} the collection $((\star_1,x_1),(\star_2,x_2))=\{(r,\mathrm{id}_{x_2}):(\star_1,x_1)\rightarrow (\star_2,x_2)\}$ Note that there will not be any arrows across indices's (1,2,3) to (4,5,6) since the two base groupoids are not connected.
	
	\item \textbf{identity}: for each object $(\star_1,x_1)$, the morphism $\mathrm{id}_{(\star_1,x_1)} = (\mathrm{id}_{\star_1},\mathrm{id}_{x_1})$
	
	\item \textbf{composition}: $(r,\mathrm{id}_{x_3}) \bullet (r,\mathrm{id}_{x_2}) = (r \circ r,\mathrm{id}_{x_3} \cdot {\mathbf{F_{2}}}r((\mathrm{id}_{x_2})) = (r^2,\mathrm{id}_{x_3})$.
	
\end{itemize}
The usual unit laws and associativity axiom are straightforward.

\begin{figure}
	\begin{equation*}
	\xymatrix{
		(\star_1,x_1) \ar@/_2pc/[dddd]_{(r^2,\mathrm{id}_{x_3})} \ar[dd]^{(r,\mathrm{id}_{x_2})} \ar@{.>}[rr]^{F_1t(\star_1,x_1)} &  & (\star_4,x_4) \ar[dd]_{(r,\mathrm{id}_{x_5})} \ar@/^2pc/[dddd]^{(r^2,\mathrm{id}_{x_6})} \\
		& & \\
		(\star_2,x_2) \ar[dd]^{(r,\mathrm{id}_{x_3})} \ar@{.>}[rr]^{F_1t(\star_2,x_2)} &  & (\star_5,x_5) \ar[dd]_{(r,\mathrm{id}_{x_6})}  \\
		& & \\
		(\star_3,x_3) \ar@{.>}[rr]_{F_1t(\star_3,x_3)} &  & (\star_6,x_6) \\
		\ast_1 \ar@(dl,ul)[]|{id_{\ast_1}} \ar[rr]_{t} &  & \ast_2 \ar@(dr,ur)[]|{id_{\ast_2}}
	}
	\end{equation*}
	\caption{Outer symmetry using groupoids}
	\label{fig:hgroupoid2}
\end{figure}

In this case the first level (outer) base $\mathbf{B}_1 = \mathcal{G}_2$, the cyclic groupoid of order two, is a two object category with two arrows and $\mathbf{F_{1}}: \mathcal{G}_2 \xrightarrow{F_1} \mathbf{Cat} \xrightarrow{I} \mathbf{Cat}$ and $\mathbf{F_{1}}(\ast_1) = (\mathcal{X}_{2a} \rtimes_{\mathbf{F_{2a}}} \mathcal{G}_3)$ while $[( \mathcal{X}_2 \rtimes_{\mathbf{F_{2}}} (\mathcal{G}_3 \amalg \mathcal{G}_3)] \rtimes_{\mathbf{F_{1}}} \mathcal{G}_2$ consists of:

\begin{itemize}
	\item \textbf{objects}: the pairs of type $(\ast_1,(\star_1,x_1))$ denoted by $\mbox{Ob}([( \mathcal{X}_2 \rtimes_{\mathbf{F_{2}}} (\mathcal{G}_3 \amalg \mathcal{G}_3)] \rtimes_{\mathbf{F_{1}}} \mathcal{G}_2)$
	
	\item \textbf{morphisms}: $[( \mathcal{X}_2 \rtimes_{\mathbf{F_{2}}} (\mathcal{G}_3 \amalg \mathcal{G}_3)] \rtimes_{\mathbf{F_{1}}} \mathcal{G}_2((\ast_1,(\star_1,x_1)),(\ast_2,(\star_5,x_5)))$ are pairs $\{(t,(r,\mathrm{id}_{x_4})):(\ast,(\star,Y))\rightarrow (\ast,(\star,Y))\}$ where $t:\ast_1 \rightarrow \ast_2 \in \mathcal{G}_2$, $(r,\mathrm{id}_{x_4}):\mathbf{F_{1}}t(\star_1,x_1) \rightarrow (\star_5,x_5)$
	
	\item \textbf{identity}: for $(\ast_1,(\star_1,x_1))$, the morphism $\mathrm{id}_{(\ast_1,\star_1,x_1)} = (\mathrm{id}_{\ast_1},(\mathrm{id}_{\star_1},\mathrm{id}_{x_1}))$
	
	\item \textbf{composition}: $(t,(r,\mathrm{id}_{x_3}) \bullet (t,(r,\mathrm{id}_{x_5}) = (t \circ t,(r,\mathrm{id}_{x_3}) \cdot \mathbf{F_{1}}t(r,\mathrm{id}_{x_5})) = (\mathrm{id}_{\ast_1},(r^2,\mathrm{id}_{x_3}))$ since $\mathbf{F_{1}}t(r,\mathrm{id}_{x_5}) = (r,\mathrm{id}_{x_2})$ 
\end{itemize}

The simple example of permutation of finite sets illustrated that hierarchy of algebraic structures could be modeled using hierarchy of appropriate base structured categories. It follows naturally that base structured categories from groupoid bases over group bases in modeling hierarchy of symmetries are preferred due to the following heuristic reasons:
\begin{itemize}
	\item The ability to incorporate local non-uniform symmetries using multi-object and arbitrarily connected groupoid base.
	\item Ease of extraction both in terms of computational efficiency, parallel processing and non-redundant processing in a groupoid base.
	\item The ability to proceed in distinct steps by separation of objects at each step is naturally provided by the hierarchy of groupoids rather than groups.    
\end{itemize}

\subsection{Connection of hierarchy of base structured categories with standard composite of Grothendieck fibrations}
The hierarchy of base structured categories are naturally connected with the classic composite fibration in a obvious manner since multi-level base structured category is a special case of composite fibration of categories.

First we recall the classic composite of fibration Lemma from \cite{BJ}.  
\begin{lemma}\cite{BJ}
	Let $P: \mathbf{E} \rightarrow \mathbf{B_2}_{fib}$ and $Q: \mathbf{B_2}_{fib} \rightarrow \mathbf{B_1}_{fib}$ be the usual fibrations; then
	1. Composite functor $QP: \mathbf{E} \rightarrow \mathbf{B_1}_{fib}$ is a usual proper fibration where $f$ in $\mathbf{E}$ is $QP$-cartesian ; $f$ is $P$-cartesian and $Pf$ is $Q$-cartesian. Further $P$ and $Q$ both being cloven or split, the composite fibration $QP$ is respectively cloven or split; and \\
	2. For every object $I \in \mathbf{B_1}_{fib}$, we obtain by restriction a functor $P_I:\mathbf{E}_I \rightarrow \mathbf{{B_2}_I}_{fib}$ where $\mathbf{E}_I = (QP)^{-1}(I)$ and $\mathbf{{B_2}_I}_{fib} = Q^{-1}(I)$. All such restriction functors $P_I$s are themselves fibrations. 
\end{lemma}
\begin{proof}
	Refer Lemma 1.5.5 of \cite{BJ}.
\end{proof}

Observe that $\mathbf{B_1}_{fib}$ of the classic case is same as $\mathbf{B}_1$ in the hierarchy of base structured categories. Now since in hierarchical base structured category; each object is itself a base structured category which is also naturally an object of $\mathbf{Cat}$ we have standard split fibration with $\mathcal{X}_{1a}$, $\mathcal{X}_{1b}$ as objects of $\mathbf{Cat}$. Thus if we denote the pseudofunctor corresponding to split fibration $QP$ by $\Psi: \mathbf{B_1}_{fib}^{op} \rightarrow \mathbf{Cat}$, then $\Psi(I) = \mathcal{X}_{1a}$ and so on. Now since each base structured category $\mathcal{X}_{1a}$ also has a corresponding base category $\mathbf{B}_{2a}$, we naturally also have a functor $F_{1B} : \mathbf{B}_1 \rightarrow \mathbf{Cat}$ with $F_{1B}(I) = \mathbf{B}_{2a}$ where $I$ is the object of $\mathbf{B}_1$ in the usual base category of classic fibration. Hence the total category $\mathbf{B_2}_{fib}$ of the classic split fibration $Q$ is given by $\mathbf{B_2}_{fib} = \mathbf{B}_{2} \rtimes_{F_1B} \mathbf{B}_1$ where $\mathbf{B}_{2} = (\mathbf{B}_{2a} \amalg \mathbf{B}_{2b} \amalg ...)$ is the coproduct of second level base categories. Thus $\mathbf{E} = [(\mathcal{X}_{2a} \rtimes_{\mathbf{F_{2a}}} \mathbf{B}_{2a}) \amalg (\mathcal{X}_{2b} \rtimes_{\mathbf{F_{2b}}} \mathbf{B}_{2b}) \amalg ...] \rtimes_{\mathbf{F_{1}}} \mathbf{B}_1$.

\section{Hierarchy of Structures: 2-groups to n-category theory}
\label{2g}
In this section we utilize the results of \cite{morton} especially that a 2-group captures the symmetry of a category and results into a transformation double category. We show a hierarchy of 2-groups capturing the symmetry of hierarchy of base structured categories as studied in Section~\ref{h} resulting into hierarchy of transformation double categories such as $(\mathcal{X}_{2a} \rtimes_{\mathbf{F_{2a}}} \mathbf{B}_{2a}) \wquot \mathbf{2G}$, $(\mathcal{X}_{2b} \rtimes_{\mathbf{F_{2b}}} \mathbf{B}_{2b}) \wquot \mathbf{2G}$ within $(\mathcal{X}_{1} \rtimes_{\mathbf{F_{1}}} \mathbf{B}_{1}) \wquot \mathbf{2G}$ naturally leading us into n-groups or higher category theory. Our primary motivation for exploring such a connection is to get some insights on structure of possible algorithms in connection with signal representation since it seems to us at this moment that such a multi-level structure is akin to multi-resolution analysis in wavelet signal representation. The precise connections will become clear only in the future.     

We start by recalling the definition of a 2-group. 

\begin{definition}\cite{morton}
	A 2-group $\mathbf{2G}$ is a 2-category with a single object ($\mathbf{2G}^{(0)} =
	\{ \star \}$), for which all 1-morphisms $\gamma \in \mathbf{2G}^{(1)}$ and
	2-morphisms $\chi \in \mathbf{2G}^{(2)}$ are invertible.
\end{definition}

This definition is studied in \cite{morton} through connection with an equivalent categorical group and crossed module. Refer to \cite{morton} and references therein for a review of symmetry using category theory or the categorification of local and global symmetry. We demonstrate the transformation double categories produced at each level by a hierarchy of 2-groups which correspondingly capture the symmetries of base structured categories at each level. This leads one into the realm of n-groups and higher categories. The outermost 2-group is denoted as in Equation~\ref{eq:outer2g} where $\gamma_o$ and $\gamma'_o$ are its 1-morphisms whereas $\chi_o$ is the 2-morphism. Similarly the inner 2-groups for the inner base structured categories are shown in Equation~\ref{eq:inner2g}.

\begin{equation}
\xymatrix{
	(\mathcal{X}_{1} \rtimes_{\mathbf{F_{1}}} \mathbf{B}_{1}) & & (\mathcal{X}_{1} \rtimes_{\mathbf{F_{1}}} \mathbf{B}_{1}) \ar@/^2pc/[ll]^{\gamma'_o}="3"  \ar@/_2pc/[ll]_{\gamma_o}="2" \\
	\ar@{=>}"2"+<0ex,-2ex> ;"3"+<0ex,+2ex>^{\chi_o}
}
\label{eq:outer2g}
\end{equation}

\begin{equation}
\xymatrix{
	(\mathcal{X}_{2a} \rtimes_{\mathbf{F_{2a}}} \mathbf{B}_{2a}) & & (\mathcal{X}_{2a} \rtimes_{\mathbf{F_{2a}}} \mathbf{B}_{2a}) \ar@/^2pc/[ll]^{\gamma'_1}="3"  \ar@/_2pc/[ll]_{\gamma_1}="2" \\
	\ar@{=>}"2"+<0ex,-2ex> ;"3"+<0ex,+2ex>^{\chi_1}
}
\xymatrix{
	(\mathcal{X}_{2b} \rtimes_{\mathbf{F_{2b}}} \mathbf{B}_{2b}) & & (\mathcal{X}_{2b} \rtimes_{\mathbf{F_{2b}}} \mathbf{B}_{2b}) \ar@/^2pc/[ll]^{\gamma'_2}="3"  \ar@/_2pc/[ll]_{\gamma_2}="2" \\
	\ar@{=>}"2"+<0ex,-2ex> ;"3"+<0ex,+2ex>^{\chi_2}
}
\label{eq:inner2g}
\end{equation}

Thus hierarchy of base structured categories naturally leads one into the domain of higher category theory (see both \cite{leinster} and \cite{baeztowards} and references therein) where each level of symmetry or base structured category has an automorphism group given by a \textbf{2-group} (or equivalently a \textbf{categorical group} or a \textbf{crossed module}) all in a hierarchical fashion. This leads to an interesting structure of a hierarchy of naturality squares within naturality squares as shown in Figure~\ref{fig:natural}. Thus a 2-level hierarchy results into 3-morphisms ($\gamma_o, \gamma'_o$) and 4-morphisms ($\chi_o$) of outer 2-group relative to 1-morphisms ($\gamma_1, \gamma'_1,$) and 2-morphisms ($\chi_1$) of inner 2-groups.

\begin{figure}
	\begin{center}
		\begin{tikzpicture}[>=latex]
		
		\foreach \x in {1.5,2.5} {
			\foreach \y in {1.5,2.5} {
				\fill (\x,\y) circle (0.07cm);
				\fill (-\x,-\y) circle (0.07cm);
				\fill (\x,-\y) circle (0.07cm);
				\fill (-\x,\y) circle (0.07cm);
				\draw[-] (1.7,\y)--(2.3,\y) node [midway,above] {$2$};
				\draw[-] (1.7,-\y)--(2.3,-\y) node [midway,below] {$2$};
				\draw[-] (-2.3,\y) -- (-1.7,\y) node [midway,above] {$2$};
				\draw[-] (-2.3,-\y) -- (-1.7,-\y)node [midway,below] {$2$};
			}
			\draw [-] (\x,1.7) --(\x,2.3) node [midway,right] {$1$};
			\draw [-] (-\x,1.7) -- (-\x,2.3)node [midway,left] {$1$};
			\draw [-] (\x,-1.7) -- (\x,-2.3) node [midway,right] {$1$};
			\draw [-] (-\x,-1.7) -- (-\x,-2.3)node [midway,left] {$1$};
		}
		
		\draw [-] (-1.3,2) -- (1.3,2) node [midway, above] {$4$};
		\draw [-] (-1.3,-2) -- (1.3,-2) node [midway, above] {$4$};
		
		\draw[-] (2,-1.3) -- (2,1.3) node [midway,left] {$3$};
		\draw[-] (-2,-1.3) -- (-2,1.3) node [midway,left] {$3$};
		
		\end{tikzpicture}
		
	\end{center}
	\caption{Hierarchy of naturality squares leading to 3-,4-morphisms}
	\label{fig:natural}
\end{figure}
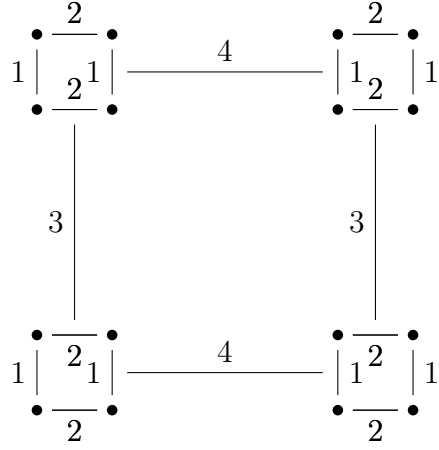

We now use the example of $[(\mathcal{X}_1 \rtimes_{\mathbf{F_{2a}}} \mathcal{G}_3) \amalg (\mathcal{X}_2 \rtimes_{\mathbf{F_{2b}}} \mathcal{G}_3)] \rtimes_{\mathbf{F_{1}}} \mathcal{G}_2$ to show how the fundamental cells or squares corresponding to $(\chi,f) \in \mathbf{2G} \times \mathbf{C}$ from \cite{morton} appear in a hierarchy starting with the inner $\mathbf{2G}$ to outer $\mathbf{2G}$.

\begin{equation}
\xymatrix{
	\mathbf{2G} \ar[r]^-{\Phi_o} \ar[rd]_{\hat{\phi}_o} & Aut(\mathcal{X} \rtimes_{\mathbf{F_{1}}} \mathcal{G}_2) \ar[d] \\
	& \mathcal{X} \rtimes_{\mathbf{F_{1}}} \mathcal{G}_2
}
\label{eq:outer2gact}
\end{equation}

Here $\hat{\phi}_o : \mathbf{2G} \times (\mathcal{X} \rtimes_{\mathbf{F_{1}}} \mathcal{G}_2) \rightarrow (\mathcal{X} \rtimes_{\mathbf{F_{1}}} \mathcal{G}_2)$ is the 2-group action. Now ${\mathbf{\Phi}_o}$ is a 2-functor from $\mathbf{2G}$ to 2-category of invertible endofunctors on category $(\mathcal{X} \rtimes_{\mathbf{F_{1}}} \mathcal{G}_2)$. The fundamental unit (or cell or naturality) square of co-domain endofunctor category is the square corresponding to a morphism $(\chi,f) \in \mathbf{2G} \times \mathbf{C}$ which is the action of $\mathbf{2G}$ on $\mathbf{C}$; see \cite{morton}.

\subsection{The naturality square corresponding to inner 2-group action}

First the symmetry of $\mathcal{X}_1 \rtimes_{\mathbf{F_{2a}}} \mathcal{G}_3$ is expressed using an (inner) 2-group as shown in Equation~\ref{eq:inner2gact}. 

\begin{equation}
\xymatrix{
	\mathbf{2G} \ar[r]^-{\Phi} \ar[rd]_{\hat{\phi}} & Aut(\mathcal{X}_1 \rtimes_{\mathbf{F_{2a}}} \mathcal{G}_3) \ar[d] \\
	& \mathcal{X}_1 \rtimes_{\mathbf{F_{2a}}} \mathcal{G}_3
}
\label{eq:inner2gact}
\end{equation}

The 2-functor $\Phi$ defines a mapping on the single object given by $\Phi(\star) = (\mathcal{X}_1 \rtimes_{\mathbf{F_{2a}}} \mathcal{G}_3)$. The mapping of 1-morphism is denoted as $\Phi(\gamma)$ while that of the 2-morphism is denoted as $\Phi(\gamma,\chi)$ where $\chi: \gamma \rightarrow \gamma'$ is the 2-morphism of the $\mathbf{2G}$.

\begin{equation}
\xymatrix{
	X \rtimes_{\mathbf{F_{2a}}} \mathcal{G}_3 & & X \rtimes_{\mathbf{F_{2a}}} \mathcal{G}_3 \ar@/^2pc/[ll]^{\Phi(\gamma')=\gamma' \act () }="3"  \ar@/_2pc/[ll]_{\Phi(\gamma)=\gamma \act () }="2" \\
	\ar@{=>}"2"+<0ex,-2ex> ;"3"+<0ex,+2ex>^{\Phi_{(\gamma,\chi)}}
}
\end{equation}

The action of 1-morphisms $\gamma,\gamma'$ and 2-morphism $\chi$ between these on a given arrow $(r,\mathrm{id}_{x_2})$ in $\mathcal{X}_1 \rtimes_{\mathbf{F_{2a}}} \mathcal{G}_3$ generates a fundamental square as shown in Equation~\ref{eq:inner}.

\begin{equation}\label{eq:inner}
\resizebox{\linewidth}{!}{
	$\xymatrix{
		(\star_1,x_1) \ar@{-->}[drrr] \ar[ddd]_{(\gamma,(\star_1,x_1))}="00" \ar@{=}[dr] \ar[rr]^{(r,\mathrm{id}_{x_2})}  & & (\star_2,x_2) \ar[ddd]^{(\gamma,(\star_2,x_2))}="10" \ar@{=}[dr] & \\
		& (\star_1,x_1)) \ar[ddd]_(.4){(\gamma',(\star_1,x_1))}="01" \ar[rr]_(.4){(r,\mathrm{id}_{x_2})}  & & (\star_2,x_2) \ar[ddd]^{(\gamma',(\star_2,x_2))}="11"\\
		& & & \\
		\gamma \act (\star_1,x_1) \ar@{-->}[drrr] \ar[dr]_{\Phi_{(\gamma,\chi)}((\star_1,x_1))} \ar[rr]^(.6){\gamma \act (r,\mathrm{id}_{x_2})} & \uurtwocell<\omit>{\omit *+[F]{(\chi,(r,\mathrm{id}_{x_2}))}} & \gamma \act (\star_2,x_2) \ar[dr]^{\Phi_{(\gamma,\chi)}((\star_2,x_2))} & \\
		& \gamma' \act (\star_1,x_1) \ar[rr]_{\gamma' \act (r,\mathrm{id}_{x_2})} & & \gamma' \act (\star_2,x_2)
	}$   
}
\end{equation}

\subsection{The naturality square at a node of outer naturality square}

Next at each node of outer naturality square as shown in Figure~\ref{fig:natural} we have inner naturality squares corresponding to inner 2-group actions. Thus the symmetry of $(\ast_1,\mathcal{X}_1 \rtimes_{\mathbf{F_{2a}}} \mathcal{G}_3)$ could be expressed using an induced (in-out) 2-group structure as shown in Equation~\ref{eq:inout2gact}.

\begin{equation}
\xymatrix{
	(\ast_1,X \rtimes_{\mathbf{F_{2a}}} \mathcal{G}_3) & & (\ast_1,X \rtimes_{\mathbf{F_{2a}}} \mathcal{G}_3) \ar@/^2pc/[ll]^{(\mathrm{id}_{\ast_1},\gamma')}="3"  \ar@/_2pc/[ll]_{(\mathrm{id}_{\ast_1},\gamma)}="2" \\
	\ar@{=>}"2"+<0ex,-2ex> ;"3"+<0ex,+2ex>^{(\mathrm{id}_{\ast_1},\chi)}
}
\label{eq:inout2gact}
\end{equation}

Corresponding to an inner natural square as shown in Equation~\ref{eq:inner} we get a natural square at the node corresponding to object $\ast_1$ of the outer groupoid $\mathcal{G}_2$ as shown in Equation~\ref{eq:inout}.

\begin{equation}\label{eq:inout}
\resizebox{\linewidth}{!}{
	$\xymatrix{
		(\ast_1,(\star_1,x_1)) \ar@{-->}[drrr] \ar[ddd]_{(\mathrm{id}_{\ast_1},(\gamma,(\star_1,x_1)))}="00" \ar@{=}[dr] \ar[rr]^{(\mathrm{id}_{\ast_1},(r,\mathrm{id}_{x_2}))}  & & (\ast_1,(\star_2,x_2)) \ar[ddd]^{(\mathrm{id}_{\ast_1},(\gamma,(\star_2,x_2)))}="10" \ar@{=}[dr] & \\
		& (\ast_1,(\star_1,x_1)) \ar[ddd]_(.4){(\gamma',(\star_1,x_1))}="01" \ar[rr]_(.4){(\mathrm{id}_{\ast_1},(r,\mathrm{id}_{x_2}))}  & & (\ast_1,(\star_2,x_2)) \ar[ddd]^{(\gamma',(\star_2,x_2))}="11"\\
		& & & \\
		(\ast_1,(\gamma \act (\star_1,x_1))) \ar@{-->}[drrr] \ar[dr]_{(\mathrm{id}_{\ast_1},(\Phi_{(\gamma,\chi)}((\star_1,x_1))))} \ar[rr]^(.6){(\mathrm{id}_{\ast_1},(\gamma \act (r,\mathrm{id}_{x_2})))} & \uurtwocell<\omit>{\omit *+[F]{(\ast_1,(\chi,(r,\mathrm{id}_{x_2})))}} & (\ast_1,(\gamma \act (\star_2,x_2))) \ar[dr]^{(\mathrm{id}_{\ast_1},(\Phi_{(\gamma,\chi)}((\star_2,x_2))))} & \\
		& (\ast_1,(\gamma' \act (\star_1,x_1))) \ar[rr]_{(\mathrm{id}_{\ast_1},(\gamma' \act (r,\mathrm{id}_{x_2})))} & & (\ast_1,(\gamma' \act (\star_2,x_2)))
	}$   
}
\end{equation}

\subsection{The naturality square corresponding to outer 2G action}

Finally the whole symmetry of $\mathcal{X} \rtimes_{\mathbf{F_{1}}} \mathcal{G}_2$ is expressed using an (global) 2-group (which truly becomes a 4-group relative to inner 2-groups) as shown in Equation~\ref{eq:outer2gact}. The 2-functor $\Phi$ defines a mapping on the single object given by $\Phi_o(\star) = (\mathcal{X} \rtimes_{\mathbf{F_{1}}} \mathcal{G}_2)$. The mapping of 1-morphism is denoted as $\Phi_o(\gamma_o)$ while that of the 2-morphism is denoted as $\Phi_o(\gamma_o,\chi_o)$ where $\chi_o: \gamma_o \rightarrow \gamma_o'$ is the 2-morphism of the outer (global) $\mathbf{2G}$.

\begin{equation}
\xymatrix{
	\mathcal{X} \rtimes_{\mathbf{F_{1}}} \mathcal{G}_2 & & \mathcal{X} \rtimes_{\mathbf{F_{1}}} \mathcal{G}_2 \ar@/^2pc/[ll]^{\Phi(\gamma'_o)=\gamma'_o \act () }="3"  \ar@/_2pc/[ll]_{\Phi(\gamma_o)=\gamma_o \act () }="2" \\
	\ar@{=>}"2"+<0ex,-2ex> ;"3"+<0ex,+2ex>^{\Phi_{(\gamma_o,\chi_o)}}
}
\end{equation}

The action of 1-morphisms $\gamma_o,\gamma'_o$ and 2-morphism $\chi_o$ between these on a given arrow $(r,\mathrm{id}_{x_2})$ in $\mathcal{X} \rtimes_{\mathbf{F_{1}}} \mathcal{G}_2$ generates a fundamental square as shown in Equation~\ref{eq:outer}.

\begin{equation}\label{eq:outer}
\resizebox{\linewidth}{!}{
	$\xymatrix{
		(\ast_1,(\star_1,x_1)) \ar@{-->}[drrr] \ar[ddd]_{(\gamma_o,(\ast_1,(\star_1,x_1)))}="00" \ar@{=}[dr] \ar[rr]^{(t,(r,\mathrm{id}_{x_4}))}  & & (\ast_2,(\star_5,x_5)) \ar[ddd]^{(\gamma_o,(\ast_2,(\star_5,x_5)))}="10" \ar@{=}[dr] & \\
		& (\ast_1,(\star_1,x_1)) \ar[ddd]_(.4){(\gamma'_o,(\ast_1,(\star_1,x_1)))}="01" \ar[rr]_(.4){(t,(r,\mathrm{id}_{x_4}))}  & & (\ast_2,(\star_5,x_5)) \ar[ddd]^{(\gamma'_o,(\ast_2,(\star_5,x_5)))}="11"\\
		& & & \\
		\gamma_o \act (\ast_1,(\star_1,x_1)) \ar@{-->}[drrr] \ar[dr]_{\Phi_{(\gamma_o,\chi_o)}((\star_1,x_1))} \ar[rr]^(.6){\gamma_o \act (t,(r,\mathrm{id}_{x_4}))} & \uurtwocell<\omit>{\omit *+[F]{(\chi_o,(t,(r,\mathrm{id}_{x_4})))}} & \gamma_o \act (\ast_2,(\star_5,x_5)) \ar[dr]^{\Phi_{(\gamma_o,\chi_o)}((\ast_2,(\star_5,x_5)))} & \\
		& \gamma'_o \act (\ast_1,(\star_1,x_1)) \ar[rr]_{\gamma'_o \act (t,(r,\mathrm{id}_{x_4}))} & & \gamma'_o \act (\ast_2,(\star_5,x_5))
	} $
}  
\end{equation}

Thus a simple example of finite sets with 2-levels of symmetry served to illustrate a nesting of 2-groups leading to a 4-group which captures its essential structure. We shall not work out any further details regarding the connection of such base hierarchy with the existing literature \cite{leinster}, \cite{baeztowards} in this paper. Understanding the complete ramification of such a connection from an applied perspective is an interesting research direction.

\section{Geometries as Base structured categories $\mathcal{X} \rtimes_{\mathbb{F}} \mathbf{C}$}
\label{sec:groupoid}
As studied comprehensively in \cite{sharpe}, Felix Klein unified many geometries utilizing the notion of a principal group. First we recall the classic definition of Klein geometry.

\begin{definition}\cite{sharpe}
	\label{kgeometry}
	A Klein geometry is a pair $(G,H)$ where $G$ is a Lie group and ${H} \subset {G}$ a closed subgroup such that ${G}/{H}$ is connected. ${G}$ is called the principal group of the geometry. The coset space ${X} = {G}/{H}$ is the space of the geometry or Klein geometry.          
\end{definition}

\begin{proposition}
\label{kgb}
Every Klein Geometry is a base-structured category.	
\end{proposition}
\begin{proof}
Consider Klein geometry which is a pair $(G,H)$ as in the definition~\ref{kgeometry}. To show that it is a base-structured category first consider the closed lie subgroup ${H} \subset {G}$. Since every group could be considered as a category with one object in which every morphism is an isomorphism; the lie group $H$ is viewed as a category $\mathbf{H}$ with a single object $\star$ with $Hom(\star,\star)= H$ and the composition map on homomorphisms is the multiplication in $H$. Note that $H$ is equivalently a group object in the category $\mathbf{Man}^{\infty}$ of smooth manifolds and therefore its homset $Hom(\star,\star)$ is an object of $\mathbf{Man}^{\infty}$ with this added structure. Now we can define a functor ${\mathbb{F}}: \mathbf{H} \xrightarrow{{F}} \mathbf{Man}^{\infty} \xrightarrow{U} \mathbf{Set}$. The Grothendieck completion of this functor gives the base structured category $\mathbf{G} = \mathcal{X} \rtimes_{\mathbb{F}} \mathbf{H}$. The image of single object $\star$ under functor $F$  is the smooth connected manifold $X = {{F}}(\star) $ which is same as the discrete category ${\mathbb{F}}(\star)$ which is the space of the geometry (which we shall call a global space to emphasize the crucial fact that it is the image of single object in the base category). The category $\mathbf{G}$ is simply isomorphic to the transformation groupoid corresponding to the action of lie group $G$ on the smooth manifold $X$ or precisely $X // G$ as defined in~\ref{def:transfngroupoid}. Hence Klein geometry $(G,H)$ is essentially a base structured category $\mathbf{G} = \mathcal{X} \rtimes_{\mathbb{F}} \mathbf{H}$.
\end{proof}

In a spirit of generalization, following the original motivation of Charles Ehrsemann \cite{charles}, \cite{andree} which is reviewed in \cite{pradines} where the term groupoid geometry is first coined as far as we are aware, we give a simple definition of \textbf{groupoid geometry} using the perspective offered by the base structured category $\mathcal{X} \rtimes_{\mathbb{F}} \mathbf{C}$. The Proposition~\ref{kgb} motivates one naturally to consider a lie groupoid replacing the classic lie group $H$. Thus we have a new base lie groupoid structured geometry generalizing the classic Klein geometry.

\begin{definition}
	\label{geometry}
	A groupoid (or base groupoid structured) geometry is a pair $(\mathcal{G},\mathcal{B})$ where $\mathcal{G}$ is a Lie groupoid and $\mathcal{B} \subset \mathcal{G}$ a closed subgroupoid such that $\mathcal{G}/\mathcal{B}$ is globally arbitrarily connected (locally connected). $\mathcal{G}$ will be called the principal groupoid of the geometry. The multi-object coproduct space $\mathcal{X} = \mathcal{G}/\mathcal{B}$ is the usual space of the geometry or (similar to Klein geometry) `groupoid geometry' itself.          
\end{definition}

The motivation comes from generalizing the geometry of the function or signal space (that corresponds to generative base structured signal spaces) from Euclidean geometry (that corresponds to Hilbert spaces)    

With this definition, it should even be possible to conceive of a geometry which contains inner hierarchy of geometries using the hierarchy of symmetries or base categories as developed in the earlier Section~\ref{h}. We now note down a quick comparison of this geometry with the existing notion of Klein geometry in Table~\ref{table}. This serves to intuitively illustrate the more general properties of this geometry with a multi-origin arbitrarily connected spaces, all attributed to the concept of multi-object groupoid action.

\begin{center}
	\begin{table}[ht]
		\centering
		\begin{tabular}{|c|c|}\hline
			\textbf{Klein Geometry} & \textbf{Groupoid Geometry} \\ \hline
			Lie Groups $(G,H)$ & Lie Groupoids $(\mathcal{G},\mathcal{B})$ \\ \hline
			Principal group $G = X \rtimes H$ & Principal groupoid $\mathcal{G} = \mathcal{X} \rtimes \mathcal{B}$ \\ \hline
			Homogeneous coset space $X = G/H$  & Coproduct Space $\mathcal{X} = X_1 \amalg X_2 \amalg ... = \mathcal{G}/\mathcal{B}$ \\ \hline
			Single basepoint(origin) $x_0 \in X$ & Multi-basepoints (origins) $x_1,x_2,.. \in \mathcal{X}$ \\ \hline
			$H$ is stabilizer subgroup of $x_0$ in $G$  & Local groups in $\mathcal{B}$ stabilize $x_1,x_2,...$ in $\mathcal{G}$ \\ \hline
		\end{tabular}
		\caption{From Klein Geometries \cite{pradines} to Groupoid Geometries}
		\label{table}
	\end{table}
\end{center}

\begin{proposition}
	\label{ggb}
	A groupoid geometry is a base-structured category $\mathbf{G} = \mathcal{X} \rtimes_{\mathbb{F}} \mathbf{B}$.	
\end{proposition}
\begin{proof}
	Consider groupoid geometry which is a pair $(\mathcal{G},\mathcal{B})$ as in the definition~\ref{geometry}. To show that it is a base-structured category first consider the closed lie subgroupoid $\mathcal{B} \subset \mathcal{G}$. Since every groupouid could be considered as a category with multiple objects in which every morphism is an isomorphism; the lie groupoid $\mathcal{B}$ is viewed as a category $\mathbf{B}$ with many objects $\star_1,\star_2,..$. Here $\mathcal{B}$ is equivalently a groupoid object in the category $\mathbf{Man}^{\infty}$ of smooth manifolds. Now we can define a functor ${\mathbb{F}}: \mathbf{B} \xrightarrow{{F}} \mathbf{Man}^{\infty} \xrightarrow{U} \mathbf{Set}$. The Grothendieck completion of this functor gives the base structured category $\mathbf{G} = \mathcal{X} \rtimes_{\mathbb{F}} \mathbf{B}$. The image of multiple objects $\star_1,\star_2,..$ under functor $F$ are multiple smooth connected manifolds $X_1 = {{F}}(\star_1), X_2 = {{F}}(\star_2),... $ which are same as the multiple discrete categories ${\mathbb{F}}(\star_1), {\mathbb{F}}(\star_2), ...$ which is collective space of the geometry or more precisely $X = \amalg_{\star \in Ob(\mathbf{B})} F(\star) = X_1 \amalg X_2 \amalg ...$ (which we shall call a collection of local spaces to emphasize the crucial fact that it is the image of multiple objects in the base category). Hence groupoid geometry $(\mathcal{G},\mathcal{B})$ is essentially a base structured category $\mathbf{G} = \mathcal{X} \rtimes_{\mathbb{F}} \mathbf{B}$ where $\mathcal{G},\mathcal{B}$ are set-theoretic notations for the small categories $\mathbf{G},\mathbf{B}$. 
\end{proof}

With limited exposure to the history of development category theory, it appears to us that the concept of groupoid geometry was first conceived by Charles Ehresmann \cite{charles} and explicitly proposed in \cite{pradines} by generalizing smooth (locally trivial) principal bundle (in Ehresmann's sense) to non-locally trivial bundles. Referring \cite{andree}, Charles Ehresmann who showed the equivalence between various notions such as classic presheaves ($\mathbf{Set}$ valued functors), discrete fibrations and (mathematical) species of structures. In this sequel we try to generalize this to $\mathbf{D}$ valued functor $F$ by using a suitable $\mathbf{F}$ or $\mathbb{F}$ for objects general than sets in connection with our applied work which utilizes Hilbert spaces as objects rather than just sets. Again referring \cite{andree}, Ehresmann reformulated the fundamentals of differential geometry specifically using concepts of fibred spaces and foliated manifolds and some sort of a species of local structures and associated a category (groupoid) to it. Starting from \cite{pradines}, we have tried to conceptualize the geometry by interpreting the work of \cite{charles} from the perspective of category action introduced in~\cite{salilp1} of this sequel. In doing so we have proposed a simple definition for a geometry as a base structured category which arises from a general category action. As far as we are aware (at the time of writing this paper), within the literature there is no such definition~\ref{geometry} of a groupoid geometry. An interested reader well-versed on the subject of differential geometry along with considerable exposure to the category $\mathbf{Diff}$; can verify and completely develop this concept of base groupoid structured geometry (with a hierarchy of geometries using the hierarchy of bases we developed here giving rise to geometries within geometries). The details of connections with existing notion of non-commutative geometries \cite{connes} will be worked out in the future. Finally the reader must note that in \cite{Leyton01} the author argues against the Klein geometry using wreath group as implied base in situations involving transfer. It seems appropriate to us that the groupoid structured geometry (with hierarchy of groupoid bases) as defined here models the original conception of a generative geometry \cite{Leyton01} using hierarchy of multi-object groupoids than single-object groups in a functorial way.

\section{Multi-object species of structures, Multi-object signal spaces: Heuristic Discussion}
\label{sec:pointedness}

In this section referring \cite{brown87} and \cite{joyalspecies} we motivate the concept of multi-object species of algebraic structures which have some sort of natural collection of many objects and possibly might be related to each other. As mentioned in \cite{brown87} this term seems to have originated during Bourbaki seminars. We understood such a species of structures by characterizing them as having a base object containing multiple objects or equivalently as a functor $F: \mathbf{C} \rightarrow \mathbf{D}$ from a base category with multiple objects to a general category generalizing the spirit and definition of species of structures as endofunctors in \cite{joyalspecies}. For this first we revisit what can be loosely termed as a sort of object-arrow duality which is implicitly present in this sequel using the notion of base structured categories that characterize a functor or in other words object (category) that essentially captures the structure of arrow (functor). Recall that $(F,\mathbf{C},\mathbf{D})$ characterizes a functor as a structure preserving morphism and captures the essential structure of $F$. As an illustration consider the special example of an identity functor $\mathrm{id}_\mathbf{C}$. This functor can be treated as a category $\mathbf{C}$ which is just reflecting the fact that every identity arrow of an object can be taken to be the object itself. However the base structured category $(\mathrm{id}_\mathbf{C},\mathbf{C},\mathbf{C})$ also serves to capture the structure of $\mathrm{id}_\mathbf{C}$. Indeed by construction this is a category which consists of,
\begin{itemize}
	\item \textbf{objects}: $(X,X),(Y,Y), ...$ denoted by $\mbox{Ob}(\mathbf{G}_{\mathrm{id}_\mathbf{C}})$
	
	\item \textbf{morphisms}: collection $\mathbf{G}_{\mathrm{id}_\mathbf{C}}((X,X),(Y,Y))=\{(f,f):(X,X)\rightarrow (Y,Y)\}$
	
	\item \textbf{identity}: for each $(X,X)$, the morphism $\mathbf{1}_{(X,X)} = (\mathbf{1}_X,\mathbf{1}_{X})$
	
	\item \textbf{composition}: if $(g,f)\mapsto g\circ f$ in $\mathbf{C}$ then $((g,g),(f,f))\mapsto (g,g)\cdot (f,f)=(g\circ f,g\circ f)$
	
	\item \textbf{unit laws}: for $(f,f)$ we have $\mathbf{1}_{(Y,Y)}\cdot (f,f)=(f,f)=(f,f)\cdot \mathbf{1}_{(X,X)}$
	
	\item \textbf{associativity}:  $(h,h)\cdot ((g,g)\cdot (f,f))=((h,h)\cdot (g,g))\cdot (f,f)=(h,h)\cdot (g,g)\cdot (f,f) $
\end{itemize}
This is isomorphic to the category $\mathbf{C}$ itself and also serves to characterize the functor $\mathrm{id}_\mathbf{C}$. This implies that if $F: \mathbf{C} \rightarrow \mathbf{D}$ is taken to be a species of structures then it is equivalently characterized by the corresponding base structured category $(F,\mathbf{C},\mathbf{D})$ with multiple objects. 

Now we illustrate the notion of multi-object algebraic structures from a slightly different perspective of internal categories. First we recall the notion of algebraic structures within an ambient category or objects internal to some ambient category from \cite{CWM} and formulate a comparative difference taking the example of a group and a groupoid internal to some fixed ambient category.

A Groupoid in a category is a tuple consisting of,
\begin{enumerate}
	\item $G^{(0)}$ Object of Objects (also called Base Object $B$), 
	\item $G^{(1)}$ Object of arrows, 
	\item $d_0$ Domain map, 
	\item $d_1$ Codomain map, 
	\item $e$ Identity (section), 
	\item $m$ multiplication, 
	\item $i$ Inverse map. 
\end{enumerate}
This 7-tuple is simply reduced to 4-tuple in the special case of a group as denoted in Equation~\ref{diff1}.

\begin{equation}
\label{diff1}
\mbox{Groupoid:} 
\quad
G= (G^{(0)},G^{(1)},d_0,d_1,e,m,i)  ;
\mbox{Group:} 
\qquad 
G=(G^{(1)},e,m,i)
\end{equation}

The base object $B$ of a groupoid, is a general object of the ambient category which introduces the concept of multiple objects or object of objects or non-pointedness and arbitrary connectedness (since these multiple objects will have arbitrary interconnecting arrows). Precisely when this base object becomes a terminal object of the ambient category, which means there are no sub objects then we recover the special case of pointed and connected groupoid which is a group.

\begin{equation}
\xymatrix{
	G^{(1)} \ar@/^1pc/[r]^{d_0} \ar@/_1pc/[r]_{d_1} & B
}
\qquad
\xymatrix{
	G^{(1)} \ar[r]^-{d_0 = d_1}_-{!} & G^{(0)}=T=\star=B
}
\end{equation}
The identity map produces a section in the case of a groupoid which naturally reduces to an element in case of a group as below:
\begin{equation}
\xymatrix{
	B \ar[r]^{e} & G^{(1)}
}
\qquad
\xymatrix{
	T \ar[r]^{e} & G^{(1)}
}
\end{equation}
In case of groupoid, the object of composable pair of arrows is precisely the fibered product on the base which reduces to a general product in the special case of the group implying all arrows can be composed.
\begin{equation}
\xymatrix{
	G^{(2)}=G^{(1)} \times_{B} G^{(1)} \ar[r]^-{p_1} \ar[d]_{p_0} & G^{(1)} \ar[d]^{d_0} \\
	G^{(1)} \ar[r]_{d_1} & B
}
\qquad
\xymatrix{
	G^{(2)}=G^{(1)} \times G^{(1)} \ar[r]^{p_1} \ar[d]_{p_0} & G^{(1)} \ar[d]^{!} \\
	G^{(1)} \ar[r]_{!} & T
}
\end{equation}
The multiplication or the composition map is given below consists of 2 squares which are reduced to triviality in case of a group.
\begin{equation}
\xymatrix{
	G^{(1)} \ar[d]_{d_0} & G^{(2)} \ar[l]^{p_0} \ar[d]^{m} \ar[r]^{p_1} & G^{(1)} \ar[d]^{d_1} \\
	B & G^{(1)} \ar[l]_{d_0} \ar[r]_{d_1} & B
}
\quad
\xymatrix{
	G^{(1)} \ar[d]_{!} & G^{(2)} \ar[l]^{p_0} \ar[d]^{m} \ar[r]^{p_1} & G^{(1)} \ar[d]^{!} \\
	T & G^{(1)} \ar[l]_{!} \ar[r]_{!} & T
}
\end{equation}
The unit laws for left and right identity in case of a group reduce to usual laws of a group.
\begin{equation}
\xymatrix{
	B \times_{B} G^{(1)} \ar[rd]_{d_0} \ar[r]^{e \times id} & G^{(2)} \ar[d]^{m} & G^{(1)} \times_{B} B \ar[l]^{id \times e} \ar[ld]^{d_1} \\
	& G^{(1)} & 
}
\quad
\xymatrix{
	B \times_{B} G^{(1)} \ar[rd]_{d_0} \ar[r]^{e \times id} & G^{(2)} \ar[d]^{m} & G^{(1)} \times_{B} B \ar[l]^{id \times e} \ar[ld]^{d_1} \\
	& G^{(1)} & 
}
\end{equation}
The composable triple of arrows for a groupoid is again 2 fibered products in serial which reduce to usual products for a group.
\begin{equation}
\xymatrix{
	G^{(3)}=G^{(1)} \times_{B} G^{(1)} \times_{B} G^{(1)} \ar[r] \ar[d] & G^{(1)} \ar[d]^{d_0} \\
	G^{(2)} \ar[r]_{d_1 \circ p_1} & B
}
\qquad
\xymatrix{
	G^{(3)}=G^{(1)} \times G^{(1)} \times G^{(1)} \ar[r] \ar[d] & G^{(1)} \ar[d]^{!} \\
	G^{(2)} \ar[r]_{!} & T
}
\end{equation}
The usual associativity axiom for these is
\begin{equation}
\xymatrix{
	G^{(3)} \ar[r]^{id \times m} \ar[d]_{m \times id} & G^{(2)} \ar[d]^{m} \\
	G^{(2)} \ar[r]_{m} & G^{(1)}
}
\qquad
\xymatrix{
	G^{(3)} \ar[r]^{id \times m} \ar[d]_{m \times id} & G^{(2)} \ar[d]^{m} \\
	G^{(2)} \ar[r]_{m} & G^{(1)}
}
\end{equation}

Hence on introspection we deduce a crucial fact that it is the notion of a non-terminal base object which gives rise to a general concept of multi-object species of structures. Some other examples of such species are presented in Table~\ref{table2}.

\begin{table}[ht]
	\centering
	\resizebox{0.85\linewidth}{!}{
		\begin{tabular}{|c|c|}\hline
			\textbf{Pointed,Connected,Single Object} & \textbf{Non-Pointed,Arbitrarily Connected, Multi-object} \\ \hline
			Monoid & Category \\ \hline
			Group & Groupoid \\ \hline
			Monoidal Category & 2-category \\ \hline
			2-group & 2-groupoid \\ \hline
			Wreath Group (Leyton's Transfer) & Semi-direct Groupoid (Ehrsemann Transport) \\ \hline
			Classical Klein Group Geometries  & Groupoid Geometries \\ \hline
			Vector Space & Vector Bundle \\ \hline
		\end{tabular}
	}
	\caption{single-object against multi-object species of structures}
	\label{table2}
\end{table}

Thus equivalent to multi-object base categories, multi-object species of structures could also be characterized as internal categories in some ambient category. For such multi-object species of structures, the inherent natural structure is best modeled using a base which is a category consisting of many objects. The object upon which this base category $\mathbf{C}$ acts is $\mathcal{X} = \amalg_{X \in Ob(\mathbf{C})} F(X)$ which is the object of objects (or a coproduct of objects) in $\mathbf{D}$ and naturally corresponds to a base object of a given multi-object species. Depending on the kind of action, one can expect to capture the whole or partial structure within the base category while the functor(arrow) represents the structure relative to a base. Note that given a multi-object species of structure it is possible to represent this structure relative to a pointed or single-object base category. As an example we have already seen earlier in Section~\ref{sec:wreath} that a wreath group captures the structure of a set consisting of multiple objects or subsets as a single object set. However such a representation compared to multi-object base category such as semidirect groupoid is highly suboptimal and not naturally matched to its true generative structure (in sense of Leyton's generative theory \cite{Leyton01}) and takes more amount of memory storage. This is studied in detail in the context of signal spaces in \cite{salilp3} and \cite{salilp4}. Classically the signal processing community has been utilizing single-object species such as vector spaces and group theoretic representations such as wavelets, gabor, shearlets etc formulated through theory of group frames see \cite{fft} and references therein until recently in \cite{guillemard} where multi-object structure (groupoid $C_{\ast}$ algebra) was explored for sparse time-frequency representations. This has motivated us to study the deeper implications of using multi-object species such as base structured categories with multiple objects in the base category as signal spaces and its impact on fundamental notions of redundancy, linear independence, true dimensionality and general information analysis which we formulate in \cite{salilp3}.

Thus most general base perspective is offered by a category which covers all the cases both single as well multi-object. This is further strengthened and generalized using a hierarchy of bases leading to multi-resolution kind of analysis on algebraic structures and corresponding n-category species of structures. From the Tables~\ref{table3},~\ref{table4},~\ref{table5} the reader can intuitively infer that relative to the choice of $B$, the structure of $X$ is resolved. This gives rise to a crucial question of how does one make appropriate choice of $B$ to be able to get to the natural structure within $X$ which has not been answered yet in this sequel. This will be dealt in \cite{salilp4} of the sequel.

\begin{center}
	\begin{table}[ht]
		\centering
		\resizebox{\linewidth}{!}{	
			\begin{tabular}{|c|c|}\hline
				\textbf{Base Category $\mathbf{C}$} & \textbf{Induced structure on $X = F(\mathbf{C})$ in $\mathbf{D}$}\\ \hline
				Singleton (Single Object, Trivial arrow) & X modeled as object in $\mathbf{D}$  \\ \hline
				Set $I$ (Multi-Object, Trivial arrows) & X modeled as coproduct object \\ \hline
				Monoid $\mathbf{M}$ (Single Object, Multi-arrow) & X modeled as monoid object \\ \hline
				Category $\mathbf{C}$ (Multi-Object, Multi-arrow) & X modeled as category object\\ \hline
			\end{tabular}
		}
		\caption{Structure induced on $X$ in a general category $\mathbf{D}$ relative to different bases.}
		\label{table3}
	\end{table}
\end{center}  

\begin{center}
	\begin{table}[ht]
		\centering
		\resizebox{\linewidth}{!}{	
			\begin{tabular}{|c|c|}\hline
				\textbf{Base Category $\mathbf{B}$} & \textbf{Induced structure on $X = F(\mathbf{B})$ in $\mathbf{Set}$}\\ \hline
				Singleton (Single Object, Trivial arrow) & X modeled as single whole set \\ \hline
				Set $I$ (Multi-Object, Trivial arrows) & X a set partitioned into subsets \\ \hline
				Monoid $\mathbf{M}$ (Single Object, Multi-arrow) & X as endoset on some set \\ \hline
				Category $\mathbf{C}$ (Multi-Object, Multi-arrow) & X as small category \\ \hline
			\end{tabular}
		}
		\caption{Structure induced on $X$ in $\mathbf{Set}$ relative to different bases.}
		\label{table4}
	\end{table}
\end{center}  

\begin{center}
	\begin{table}[ht]
		\centering
		\resizebox{\linewidth}{!}{	
			\begin{tabular}{|c|c|}\hline
				\textbf{Base Category as Structure} & \textbf{Induced structure on $X = F(\mathbf{B})$ in $\mathbf{Cat}$}\\ \hline
				Terminal (Single Object, Trivial arrow) & X as small 1-category\\ \hline
				Discrete (Multi-Object, Trivial arrows) & X as coproduct category \\ \hline
				Monoid $\mathbf{M}$ (Single Object, Multi-arrow) & X as endofunctor category,(strict)Monoidal category \\ \hline
				Category $\mathbf{C}$ (Multi-Object, Multi-arrow) & X as 2-category \\ \hline
			\end{tabular}
		}
		\caption{Structure induced on $X$ in $\mathbf{Cat}$ relative to different bases.}
		\label{table5}
	\end{table}
\end{center} 

In summary, the heuristic conclusions which could be drawn based on the discussion in this section are as follows:
\begin{itemize}
	\item
	Group theory capturing symmetry structure in applications when seen through the lens of category theory is always relative to a trivial base or precisely in a given ambient category, the terminal object with no sub-objects. Groupoid theory again capturing symmetry could be considered as a generalization of group theory which is generalized to an arbitrary base. The base is precisely any general object containing multiple sub-objects and need not be terminal in a given ambient category.
	\item
	Groups and more generally single-object species of structures are the algebraic structures formed by collections of transformations of global nature. Breaking global transformations into local transformations leads naturally to multi-object species of structures. Intuitively it can be seen that since group theory (or theory of any single-object species of structure) looks at an object of interest from global or external point of view, to cover all local structural possibilities (essentially when whole object is naturally not rigid relative to parts of it), the theory needs to account for this from a global perspective. The single object perspective becomes even more restrictive when the base object consists of internal arrows among sub-objects having structures coming from some external category. To cover such a possibility one needs to a priori start with a reference frame which is in some sense fibered on other reference frame. This crucial dependence of frame of reference on natural object structure makes whole base structured categories more natural. Clearly multi-object base structured category perspective is optimal in all such applied situations where a complex object is naturally evolved from smaller prototype using the concept of Leyton transfer or equivalently Ehresmann transport since it accounts for arbitrary local possibilities quite naturally, coming from its non-pointedness and arbitrary connectedness.       
\end{itemize}

\section{Conclusion and Extension}
This paper expanded the base structured category $\mathcal{X} \rtimes_{\mathbf{F}} \mathbf{C}$ and $\mathcal{X} \rtimes_{\mathbb{F}} \mathbf{C}$, developing the representation perspective of a functor as a multi-object category action introduced in \cite{salilp1}. In summary:

\begin{itemize}
	\item First we looked at how Leyton's generative theory of shape could be more faithfully modeled using hierarchy of base structured categories using multi-object groupoid bases as compared to groups.
	
	\item Next using simple example of permutation action on a finite set at two levels (local and global) we showed that hierarchy of base structured categories is precisely the special case of the theory of composite Grothendieck fibrations.
	
	\item Further we utilized the notion of 2-group action on a category to propose a hierarchy of 2-groups capturing the symmetry of hierarchy of base structured categories naturally connecting with the higher category theory.
	
	\item Finally we proposed a simple Definition~\ref{geometry} of groupoid structured geometry and briefly commented on multi-object species of structures in anticipation of the impact of such structures in the applied context of multi-object generative signal spaces.
\end{itemize} 

The paper from a geometrical viewpoint illustrated the framework that we are developing beginning with \cite{salilp1}. In the next papers \cite{salilp3},\cite{salilp4} of this sequel we shall work out a definition of redundancy using the base structured categorical framework. This includes detailed discussions on compression and information analysis. We shall also discuss a few practical algorithms to learn or estimate the optimal base category for a given signal to be represented along with limitations and open issues of the proposed novel framework. 

\bibliography{p2.bib}
\bibliographystyle{acm}

\end{document}